\documentclass[10pt]{amsart}
\usepackage{amsxtra, amsfonts, amsmath, amsthm, amstext, amssymb, amscd, mathrsfs, verbatim, color}
\usepackage{threeparttable}
\usepackage{mathtools}
\usepackage[ansinew]{inputenc}\usepackage[T1]{fontenc}
\usepackage[all,cmtip]{xy}
\usepackage{tikz}

\usetikzlibrary {positioning}
\date{\today}

\theoremstyle{plain}

\newtheorem{theorem}{Theorem}[section]
\newtheorem{lemma}[theorem]{Lemma}
\newtheorem{definition-theorem}[theorem]{Definition-Theorem}
\newtheorem{proposition}[theorem]{Proposition}
\newtheorem{corollary}[theorem]{Corollary}

\theoremstyle{definition}
\newtheorem{definition}[theorem]{Definition}
\newtheorem{example}[theorem]{Example}
\newtheorem{remark}[theorem]{Remark}
\newtheorem{notation}[theorem]{Notation}
\newcommand \bth[1] { \begin{theorem}\label{t#1} }
\newcommand \ble[1] { \begin{lemma}\label{l#1} }

\newcommand \bpr[1] { \begin{proposition}\label{p#1} }
\newcommand \bco[1] { \begin{corollary}\label{c#1} }
\newcommand \bde[1] { \begin{definition}\label{d#1}\rm }
\newcommand \bex[1] { \begin{example}\label{e#1}\rm }
\newcommand \bre[1] { \begin{remark}\label{r#1}\rm }

\newcommand \bnota[1] {\begin{notation}\label{n#1}\rm }
\newcommand {\ele} { \end{lemma} }

\newcommand {\epr} { \end{proposition} }
\newcommand {\eco} { \end{corollary} }
\newcommand {\ede} { \end{definition} }
\newcommand {\eex} { \end{example} }
\newcommand {\ere} { \end{remark} }
\newcommand {\enota} { \end{notation} }

\begin{document}

\title[Derivatives]{Construction of simple quotients of Bernstein-Zelevinsky derivatives and  highest derivative multisegments II: Minimal sequences} 

\author[Kei Yuen Chan]{Kei Yuen Chan}
\address{
Department of Mathematics, The University of Hong Kong
}

\email{kychan1@hku.hk}

\maketitle

\begin{abstract}
Let $F$ be a non-Archimedean local field. For any irreducible smooth representation $\pi$ of $\mathrm{GL}_n(F)$ and a multisegment $\mathfrak m$, we have an operation $D_{\mathfrak m}(\pi)$ to construct a simple quotient $\tau$ of a Bernstein-Zelevinsky derivative of $\pi$. This article continues the previous one to study the following poset
\[  \mathcal S(\pi, \tau) :=\left\{ \mathfrak n : D_{\mathfrak n}(\pi)\cong \tau \right\} ,
\]
where $\mathfrak n$ runs for all the multisegments. Here the partial ordering on $\mathcal S(\pi, \tau)$ comes from the Zelevinsky ordering. We show that the poset has a unique minimal multisegment. Along the way, we introduce two new ingredients: fine chain orderings and local minimizability. 
\end{abstract}

\section{Introduction}

The Bernstein-Zelevinsky (BZ) derivative is a classical tool in studying the representation theory of $p$-adic general linear groups, originally introduced in \cite{BZ77} and \cite{Ze80}. In \cite{Ch22+d}, we transfer several problems concerning simple quotients into sequences of derivatives of essentially square-integrable representations, aiming to obtain specific information about Jacquet modules. In this article, we further investigate properties of these sequences of derivatives. In the subsequent work \cite{Ch22+e}, we explore additional properties of minimal sequences.

One classical viewpoint on Jacquet modules-and their adjoint functor, parabolic induction-is through the lens of Hopf algebra theory, as discussed in \cite{HM08}. This approach has proven to be very successful and useful, but it is typically conducted at the level of the Grothendieck group. The semisimplification of Jacquet modules sometimes has a more straightforward formula, as noted in \cite{Ze80} and \cite{Ma13, MT15} (for other classical groups). However, it can also be challenging in certain special cases, as demonstrated in \cite{De23} and \cite{Ja07}.

On the other hand, the irreducible quotients and submodules of Jacquet modules can serve as simpler objects for study and are also valuable for our applications. Our notion of derivatives of essentially square-integrable representations focuses on examining some specific irreducible quotients of the Jacquet module that arise from a maximal parabolic subgroup of a general linear group. For a detailed description of derivatives, see Section \ref{ss main results} below.

The main problem addressed in this article is as follows: given a collection of nice sequences of derivatives of essentially square-integrable representations, we ask whether it is possible to identify a canonical choice of an element among those that produce the same derivative. Here, the term "nice" refers to sequences that arise from Bernstein-Zelevinsky derivatives. The results concerning these simple quotients can be viewed as a generalization of the simpler cases of ladder representations discussed in \cite{LM14} and generic representations explored in \cite{Ch21}.

The study of simple quotients of BZ derivatives is also related to a branching law problem of a sign representation for affine Hecke algebras of type A (see \cite[Appendix]{Ch22+d}). Some other applications such as on the Harish-Chandra modules for $\mathrm{GL}_n(\mathbb C)$ will be explored elsewhere, see e.g. \cite{CW25}.

\subsection{Notations}
Let $G_n=\mathrm{GL}_n(F)$, the general linear group over a non-Archimedean local field $F$. The representation theory of $G_n$ over a non-Archimedean local field has a rich structure, particularly concerning its irreducible representations. A significant advancement in understanding these representations was made by Zelevinsky, who introduced multisegments to parametrize the irreducible representations of $G_n$. In this context, we fix an irreducible cuspidal representation $\rho$ of $G_r$ (for some $r$). We mainly consider Bernstein components associated with so-called simple types, as defined in \cite{BK93}, and indeed Bernstein components for different $\rho$ are closely related.

\begin{itemize} 
\item Let $\nu: G_n \rightarrow \mathbb C^{\times}$ be the character $\nu(g)=|\mathrm{det}(g)|_F$, where $|.|_F$ is the norm for $F$. 
\item For $a,b \in \mathbb Z$ with $b-a \in \mathbb Z_{\geq 0}$, we call 
\begin{align} \label{eqn seg def}
[a,b]_{\rho}:=\left\{ \nu^a\rho, \ldots, \nu^b\rho \right\} 
\end{align}
to be a {\it segment}. 

We also set $[a,a-1]_{\rho}=\emptyset$ for $a \in \mathbb Z$. For a segment $\Delta=[a,b]_{\rho}$, we write $a(\Delta)=\nu^a\rho$ and $b(\Delta)=\nu^b\rho$. We also write: \[[a]_{\rho} :=[a,a]_{\rho}, \]
which is called a singleton segment. We may also write $[\nu^a\rho, \nu^b\rho]$ for $[a,b]_{\rho}$ and write $[\nu^a\rho]$ for $[a]_{\rho}$. When we talk about length of a segment, we usually refer to the quantity $b-a+1$ for a segment $[a,b]_{\rho}$.
\item Let $\mathrm{Seg}_{\rho}$ be the set of all segments. We also consider the empty set $\emptyset$ to be in $\mathrm{Seg}_{\rho}$.
\item A {\it multisegment} is a multiset of non-empty segments. Let $\mathrm{Mult}_{\rho}$ be the set of all multisegments. We also consider the empty set $\emptyset$ to be also in $\mathrm{Mult}_{\rho}$. 
\end{itemize}

We now introduce more notions to describe properties of segments and multisegments. The more significant one is the intersection-union process, which describes possible composition factors appearing in a so-called standard representation \cite{Ze80}. The ordering is also compatible with the closure relation of the geometry in \cite{Ze81}, while we do not require such geometric interpretation in this work.
\begin{itemize}
\item Two segments $\Delta$ and $\Delta'$ are said to be {\it linked} if $\Delta \cup \Delta'$ is still a segment, and $\Delta \not\subset \Delta'$ and $\Delta' \not\subset \Delta$. Otherwise, it is called to be not linked or unlinked.
\item For $\rho_1, \rho_2 \in \mathrm{Irr}^c$, we write $\rho_2 < \rho_1$ if $\rho_1 \cong \nu^a \rho_2$ for some integer $a > 0$. For two segments $\Delta_1, \Delta_2$, we write $\Delta_1 <\Delta_2$ if $\Delta_1$ and $\Delta_2$ are linked and $b(\Delta_1)< b(\Delta_2)$. 
\item For two multisegments $\mathfrak m$ and $\mathfrak n$, we write $\mathfrak m+\mathfrak n$ to be the union of two multisegments, counting multiplicities. For a multisegment $\mathfrak m$ and a segment $\Delta$, $\mathfrak m+\Delta=\mathfrak m+\left\{\Delta \right\}$ if $\Delta$ is non-empty; and $\mathfrak m+\Delta=\mathfrak m$ if $\Delta$ is empty. The notions $\mathfrak m-\mathfrak n$ and  $\mathfrak m-\Delta$ are defined in a similar way. 
\item For two segments $\Delta$ and $\Delta'$, we write $\Delta \cup \Delta'$ and $\Delta \cap \Delta'$ for their set-theoretic union and intersection respectively.	
\item We say that a multisegment $\mathfrak n$ is obtained from $\mathfrak m$ by an {\it elementary intersection-union process/operation} if 
\[  \mathfrak n=\mathfrak m-\Delta-\Delta'+\Delta\cup \Delta'+\Delta \cap \Delta'
\]
for a pair of linked segments $\Delta$ and $\Delta'$ in $\mathfrak m$.
\item For two multisegments $\mathfrak m_1$ and $\mathfrak m_2$, write $\mathfrak m_2 \leq_Z \mathfrak m_1$ if $\mathfrak m_2$ can be obtained by a sequence of elementary intersection-union operations from $\mathfrak m_1$ (in the sense of \cite{Ze80}, see \cite{Ch22+d}) or $\mathfrak m_1=\mathfrak m_2$. In particular, if any pair of segments in $\mathfrak m$ is unlinked, then $\mathfrak m$ is a minimal element in $\mathrm{Mult}_{\rho}$ under $\leq_Z$. We shall equip $\mathrm{Mult}_{\rho}$ with the poset structure by $\leq_Z$. 
\item For a segment $\Delta=[a,b]_{\rho}$, define ${}^-\Delta=[a+1,b]_{\rho}$. For a multisegment $\mathfrak m$, define 
\[   {}^-\mathfrak m=\left\{ {}^-\Delta: \Delta \in \mathfrak m, \Delta \mbox{ is not a singleton} \right\} \quad \mbox{(counting multiplicities)} .
\]
\end{itemize}

We finally explain some representation-theoretic notions.
\begin{itemize}
\item For each segment $\Delta$, we shall denote by $\mathrm{St}(\Delta)$ the corresponding essentially square-integrable representation \cite{Ze80}. 
\item For any smooth representation $\pi_1$ of $G_{n_1}$ and smooth representation $\pi_2$ of $G_{n_2}$, define $\pi_1\times \pi_2$ to be the normalized parabolic induction.
\item Let $\mathrm{Irr}_{\rho}(G_n)$ be the set of all irreducible representations of $G_n$ which are irreducible quotients of $\nu^{a_1}\rho \times \ldots \times \nu^{a_k}\rho$, for some integers $a_1, \ldots, a_k \in \mathbb Z$. Let $\mathrm{Irr}_{\rho}=\sqcup_n \mathrm{Irr}_{\rho}(G_n)$. As mentioned above, representations in $\mathrm{Irr}_{\rho}(G_n)$ are in a Bernstein component corresponding to a simple type. See \cite[Section 3.8]{Ch22+d} for more discussions on this issue.
\end{itemize}

\subsection{Main results} \label{ss main results}

Let $N_i \subset G_n$ (depending on $n$) be the unipotent radical containing matrices of the form $\begin{pmatrix} I_{n-i} & * \\ & I_i \end
{pmatrix}$. For a smooth representation $\pi$ of $G_n$, we write $\pi_{N_i}$ to be its normalized Jacquet module associated to $N_i$. 

For $\pi \in \mathrm{Irr}_{\rho}(G_n)$ and a segment $\Delta=[a,b]_{\rho}$, there is at most one irreducible module $\tau \in \mathrm{Irr}_{\rho}(G_{n-i})$ such that 
 \[\tau \boxtimes \mathrm{St}(\Delta) \hookrightarrow \pi_{N_{b-a+1}} .
\]
 If such $\tau$ exists, we denote such $\tau$ by $D_{\Delta}(\pi)$. Otherwise, we set $D_{\Delta}(\pi)=0$. We shall refer $D_{\Delta}$ to be a {\it derivative}. Let $\varepsilon_{\Delta}(\pi)$ be the largest integer $k$ such that $(D_{\Delta})^k(\pi)\neq 0$.

A sequence of segments $[a_1,b_1]_{\rho}, \ldots, [a_k,b_k]_{\rho}$ (where $a_j, b_j\in \mathbb Z$) is said to be in an {\it ascending order} if for any $i< j$, either $[a_i,b_i]_{\rho}$ and $[a_j,b_j]_{\rho}$ are unlinked; or $a_i<a_j$. For a multisegment $\mathfrak n \in \mathrm{Mult}_{\rho}$, we write the segments in $\mathfrak n$ in an ascending order $\Delta_1, \ldots, \Delta_k$. Define 
\[  D_{\mathfrak n}(\pi):=D_{\Delta_k}\circ \ldots \circ D_{\Delta_1}(\pi) .
\]
The derivative is independent of the ordering of an ascending sequence \cite{Ch22+d}. In particular, one may choose the ordering such that $a_1\leq \ldots \leq a_k$. We say that $\mathfrak n$ is {\it admissible} to $\pi$ if $D_{\mathfrak n}(\pi)\neq 0$. We refer the reader to \cite{Mi09, LM16, Ch22+d} (and references therein) for more theory on derivatives.

For $\pi \in \mathrm{Irr}_{\rho}$, denote its $i$-th Bernstein-Zelevinsky derivative by $\pi^{(i)}$ (see \cite{Ze80, Ch22+d} for precise definitions and we shall not need this in this sequel). For a simple quotient $\tau$ of $\pi^{(i)}$, define
\[ \mathcal S(\pi, \tau) := \left\{ \mathfrak n \in \mathrm{Mult}_{\rho}: D_{\mathfrak n}(\pi) \cong \tau \right\} .
\]
The ordering  $\leq_Z$ induces a partial ordering on $\mathcal S(\pi, \tau)$, and we shall regard $\mathcal S(\pi, \tau)$ as a poset.

In \cite{Ch22+d}, we have introduced a combinatorial process, called {\it removal process} (see section \ref{ss removal process}), in studying the effect of $D_{\Delta}$. Two applications of removal process are given below:

\begin{theorem} (=Theorem \ref{thm closed under zelevinsky}) \label{cor closedness property}
Let $\pi \in \mathrm{Irr}_{\rho}$. Let $\tau$ be a simple quotient of $\pi^{(i)}$ for some $i$. If $\mathfrak n_1, \mathfrak n_2 \in \mathcal S(\pi, \tau)$ and $\mathfrak n_1 \leq_Z \mathfrak n_2$, then any $\mathfrak n_3 \in \mathrm{Mult}_{\rho}$ satisfying $\mathfrak n_1\leq_Z \mathfrak n_3\leq_Z \mathfrak n_2$ is also in $\mathcal S(\pi, \tau)$. 
\end{theorem}

In other words, $\mathcal S(\pi, \tau)$ is {\it convex} in the sense of \cite[Section 3.1]{St12}.

\begin{theorem} (=Theorem \ref{thm unique module}) \label{thm unique minimal}
Let $\pi \in \mathrm{Irr}_{\rho}$. Let $\tau$ be a simple quotient of $\pi^{(i)}$ for some $i$. If $\mathcal S(\pi, \tau)\neq \emptyset$, then $\mathcal S(\pi, \tau)$ has a unique minimal element with respect to $\leq_Z$. 
\end{theorem}

The condition $\mathcal S(\pi, \tau)\neq \emptyset$ is removed in \cite{Ch22+b}. As shown in \cite{Ch21, Ch22+d}, the socle of $\pi^{(i)}$ is multiplicity-free, and the remaining problem is to describe which representations appear in the socle of $\pi^{(i)}$. The minimal sequences give a canonical way to parametrize those representations, and in other words, we have a one-to-one correspondence between the set of admissible minimal sequences to $\pi$ and the set of simple quotients of Bernstein-Zelevinsky derivatives of $\pi$.

For the two segment case, there are some other criteria for a minimal sequence such as the non-overlapping property and a condition in terms of $\eta_{\Delta}$ in Section \ref{s dagger property}. Such case is particularly useful if one wants to "split" some segments to carry out some inductive arguments. We remark that there is no uniqueness for maximal elements in general. We give an example in Section \ref{no unique max element}.

\subsection{Methods and sketch of proofs}

As mentioned above, the main tool is the removal process, and \cite[Theorem 10.2]{Ch22+d} (see Theorem \ref{thm isomorphic derivatives}) reduces to analyse combinatorics in the removal process. Let us explain a bit more how this helps to prove Theorems \ref{cor closedness property} and \ref{thm unique minimal}. 

For Theorem \ref{cor closedness property}, we use an inductive way to analyse the choices of segments in the removal process. To encode such information, we shall use a notion of fine chains (Section \ref{s intersection union derivative closed}), and its connection to the removal process is shown in Lemma \ref{lem coincide lemma}. Such fine chain is shown to possess a nice ordering compatible with the Zelevinsky ordering in Lemma \ref{lem intersection union fine sequ comp}, and this in turn allows one to deduce Theorem \ref{cor closedness property}.

For Theorem \ref{thm unique minimal}, we single out a combinatorial criterion called local minimizability in Definition \ref{def minimizable function} that allows one to study in certain inductive way. Lemmas \ref{lem inductive finding intersection} and \ref{lem maxmize intersect} allow one to find a segment to do intersection-union process without changing the resultant multisegment (the multisegment obtained by the removal process) under the local minimizability. Lemma \ref{lem maximzable and resultant} tells if one can obtain a multisegment $\mathfrak n'$ from the intersection-union process of a multisegment $\mathfrak n$ which does not change the resultant multisegment in the removal process (i.e. $\mathfrak r(\mathfrak n, \mathfrak h)=\mathfrak r(\mathfrak n',\mathfrak h)$ for some $\mathfrak h$, and see Section \ref{ss removal process} below for the notion $\mathfrak r(.,.)$), then one can obtain the local minimizability condition under some terms in the fine chain.


\subsection{Organization of this article}

Section \ref{s highest derivative removal} recalls results on highest derivative multisegments and removal processes established in \cite{Ch22+d}. Section \ref{s intersection union derivative closed} defines a notion of fine chains and fine chain orderings to facilitate comparisons with the Zelevinsky ordering. Section \ref{s closure intersect union} shows the convexity property for $\mathcal S(\pi, \tau)$. In Section \ref{ss minimizable function}, we shall introduce a notion of local minimizability, used to show the uniqueness of a minimal element in Section \ref{s unique minimal}. Section \ref{s example minimal} gives two examples of the unique minimal elements. Section \ref{no unique max element} gives an example which uniqueness of $\leq_Z$-maximality fails. Section \ref{s dagger property} studies equivalent conditions for minimality in two segment cases.

\subsection{Acknowledgements}

The author is grateful to reviewers for providing comments that largely improve the readability of the article. This project is supported by the Research Grants Council of the Hong Kong Special Administrative Region, China (Project No: 17305223, 17308324) and NSFC grant for Excellent Young Scholar (Hong Kong and Macau) (Project No.: 12322120). This manuscript has no associated data.

\section{Highest derivative multisegments and removal process} \label{s highest derivative removal}

In this section, we recall some results in \cite{Ch22+d}. 

\subsection{More notations on multisegments} \label{ss more notations on multiseg}

For an integer $c$, let $\mathrm{Mult}_{\rho,c}^a$ be the subset of $\mathrm{Mult}_{\rho}$ containing all multisegments $\mathfrak m$ such that any segment $\Delta$ in $\mathfrak m$ satisfies $a(\Delta) \cong \nu^c\rho$. The upperscript $a$ in $\mathrm{Mult}_{\rho,c}^a$  is for $a(\Delta)$.

For a multisegment $\mathfrak m$ in $\mathrm{Mult}_{\rho}$ and an integer $c$, let $\mathfrak m[c]$ be the submultisegment of $\mathfrak m$ containing all the segments $\Delta$ satisfying $a(\Delta)\cong \nu^c\rho$.

Fix an integer $c$. Let $\Delta_1=[c,b_1]_{\rho}, \Delta_2=[c,b_2]_{\rho}$ be two non-empty segments. We write $\Delta_1 \leq_c^a \Delta_2$ if $b_1 \leq b_2$, and write $\Delta_1 <_c^a \Delta_2$ if $b_1<b_2$.

For non-empty $\mathfrak m_1, \mathfrak m_2$ in $\mathrm{Mult}_{\rho, c}^a$, label the segments in $\mathfrak m_1$ as: $\Delta_{1,k} \leq_c^a   \ldots \leq_c^a \Delta_{1,2} \leq_c^a \Delta_{1,1}$ and label the segments in $\mathfrak m_2$ as: $\Delta_{2,r} \leq_c^a \ldots  \leq_c^a \Delta_{2,2}  \leq_c \Delta_{2,1}$. We define the lexicographical ordering: $\mathfrak m_1 \leq_c^a \mathfrak m_2$ if $k \leq r$ and, for any $i \leq k$, $\Delta_{1,i} \leq_c^a \Delta_{2,i}$. We write $\mathfrak m_1 <_c^a \mathfrak m_2$ if $\mathfrak m_1 \leq_c^a \mathfrak m_2$ and $\mathfrak m_1\neq \mathfrak m_2$.

\subsection{Highest derivative multisegments} \label{ss highest derivative mult}

A multisegment $\mathfrak m$ is said to be {\it at the point $\nu^c\rho$} if any segment $\Delta$ in $\mathfrak m$ takes the form $[c,b]_{\rho}$ for some $b \geq c$. For $\pi \in \mathrm{Irr}_{\rho}$, define $\mathfrak{mxpt}^a(\pi, c)$ to be the maximal multisegment in $\mathrm{Mult}^{a}_{\rho,c}$ such that $D_{\mathfrak{mxpt}^a(\pi,c)}(\pi)\neq 0$. The maximality is determined by the ordering $\leq^a_c$ defined above. Define the {\it highest derivative multisegment} of $\pi \in \mathrm{Irr}_{\rho}$ to be
\[  \mathfrak{hd}(\pi):=\sum_{c\in \mathbb Z} \mathfrak{mxpt}^a(\pi, c).
\]
It is shown in \cite{Ch22+d} that $D_{\mathfrak{hd}(\pi)}(\pi)$ is the highest derivative of $\pi$ in the sense of \cite{Ze80}.

\subsection{Removal process} \label{ss removal process}

We write $[a,b]_{\rho} \prec^L [a',b']_{\rho}$ if either $a<a'$; or $a=a'$ and $b<b'$. Here $L$ means to compare on the 'left' value $a$ and we avoid to further use $a$ for confusing with previous notations. A non-empty segment $\Delta=[a,b]_{\rho}$ is said to be {\it admissible} to a multisegment $\mathfrak h$ if there exists a segment of the form $[a,c]_{\rho}$ in $\mathfrak h$ for some $c\geq b$.  We now recall the removal process.

An {\it infinity multisegment}, denoted by $\infty$, is just a symbol which will be used to indicate the situation of non-admissibility. 

\begin{definition} \cite[Section 8.2]{Ch22+d} \label{def removal process}
Let $\mathfrak h\in \mathrm{Mult}_{\rho}$ and let $\Delta=[a,b]_{\rho}$ be admissible to $\mathfrak h$. The {\it removal process} on $\mathfrak h$ by $\Delta$ is an algorithm to carry out the following steps:
\begin{enumerate}
\item Pick a shortest segment $[a,c]_{\rho}$ in $\mathfrak h[a]$ satisfying $b \leq c$. Set $\Delta_1=[a,c]_{\rho}$. Set $a_1=a$ and $b_1=c$.
\item One recursively finds the $\prec^L$-minimal segment $\Delta_i=[a_i,b_i]_{\rho}$ in $\mathfrak h$ such that $a_{i-1}<a_i$ and $b\leq b_i<b_{i-1}$. The process stops if one can no longer find those segments.
\item Let $\Delta_1, \ldots, \Delta_r$ be all those segments. For $1\leq i<r$, define $\Delta_i^{tr}=[a_{i+1}, b_i]_{\rho}$ and $\Delta_r^{tr}=[b+1,b_r]_{\rho}$ (possibly empty).
\item Define
\[  \mathfrak r(\Delta, \mathfrak h):= \mathfrak h-\sum_{i=1}^r\Delta_i+\sum_{i=1}^r \Delta_i^{tr} .
\]
\end{enumerate} 
We call $\Delta_1, \ldots, \Delta_r$ to be the {\it removal sequence} for $(\Delta, \mathfrak h)$. We also define $\Upsilon(\Delta, \mathfrak h)=\Delta_1$, the first segment in the removal sequence. The first segment in the removal sequence will be used few times later in some inductive proofs.

One useful property of the removal sequence is that $\Delta_1 \supset \Delta_2 \supset \ldots \supset \Delta_r$, which we shall refer to the {\it nesting property} (of the removal sequence). 

If $\Delta$ is not admissible to $\mathfrak h$, we set $\mathfrak r(\Delta, \mathfrak h)=\infty$, the infinity multisegment. We also set $\mathfrak r(\Delta, \infty)=\infty$. 
\end{definition}

Some examples of removal process are provided in \cite[Section 8]{Ch22+d}.

\begin{remark}
A multisegment $\mathfrak h$ is called {\it generic} if any two segments in $\mathfrak h$ are unlinked. The special feature in this case is that $\mathrm{St}(\mathfrak h)$ is generic and $\mathfrak{hd}(\mathrm{St}(\mathfrak h))=\mathfrak h$. In such case, for $\Delta \in \mathrm{Seg}_{\rho}$, it is shown in \cite{Ch21} that $D_{\Delta}(\pi)$ is generic. On the other hand, $\mathfrak r(\Delta, \mathfrak h)$ coincides with the generic multisegment which has the same cuspidal support as $D_{\Delta}(\pi)$. 

\end{remark}

\subsection{Computations on removal process}

We recall the following properties for computations:

\begin{lemma} \label{lem removal process} \cite{Ch22+d}
Let $\mathfrak h \in \mathrm{Mult}_{\rho}$ and let $\Delta, \Delta' \in \mathrm{Seg}_{\rho}$ be admissible to $\mathfrak h$. Then the followings hold:
\begin{enumerate}
\item \cite[Lemma 8.7]{Ch22+d} Let $\mathfrak h^*=\mathfrak h-\Upsilon(\Delta, \mathfrak h)+{}^-\Upsilon(\Delta, \mathfrak h)$. Then $\mathfrak r(\Delta, \mathfrak h)=\mathfrak r({}^-\Delta, \mathfrak h^*)$.
\item \cite[Lemma 8.8]{Ch22+d} Write $\Delta=[a,b]_{\rho}$. For any $a'<a$, $\mathfrak r(\Delta, \mathfrak h)[a']=\mathfrak h[a']$.
\item \cite[Lemma 8.9]{Ch22+d} If $\Delta \in \mathfrak h$, then $\mathfrak r(\Delta, \mathfrak h)=\mathfrak h-\Delta$.
\item \cite[Lemma 8.10]{Ch22+d} Suppose $a(\Delta)=a(\Delta')$. Then 
\[  \Upsilon(\Delta, \mathfrak h)+\Upsilon(\Delta', \mathfrak r(\Delta, \mathfrak h))=\Upsilon(\Delta', \mathfrak h)+\Upsilon(\Delta, \mathfrak r(\Delta', \mathfrak h)) .
\]
\item \cite[Lemma 8.12]{Ch22+d} If $\Delta, \Delta'$ are unlinked, then $\mathfrak r(\Delta', \mathfrak r(\Delta, \mathfrak h))=\mathfrak r(\Delta, \mathfrak r(\Delta', \mathfrak h))$.
\end{enumerate}
\end{lemma}

\subsection{Removal process for multisegments}

For $\mathfrak h \in \mathrm{Mult}_{\rho}$, and a multisegment $\mathfrak m=\left\{ \Delta_1, \ldots, \Delta_r\right\} \in \mathrm{Mult}_{\rho}$ written in an ascending order (defined in Section \ref{ss main results}), define:
\[   \mathfrak r(\mathfrak m, \mathfrak h)= \mathfrak r(\Delta_r, \ldots , \mathfrak r(\Delta_1, \mathfrak h)\ldots ) .  \]
We say that $\mathfrak m$ is {\it admissible} to $\mathfrak h$ if $\mathfrak r(\mathfrak m, \mathfrak h)\neq \infty$. 

For $\pi \in \mathrm{Irr}_{\rho}$ and $\mathfrak m \in \mathrm{Mult}_{\rho}$, define $\mathfrak r(\mathfrak m, \pi):=\mathfrak r(\mathfrak m, \mathfrak{hd}(\pi))$. The relation to derivatives is the following:

\begin{theorem} \cite[Theorem 10.2]{Ch22+d} \label{thm isomorphic derivatives} 
Let $\pi \in \mathrm{Irr}_{\rho}$. Let $\mathfrak m, \mathfrak m' \in \mathrm{Mult}_{\rho}$ be admissible to $\pi$. Then $\mathfrak m, \mathfrak m'$ are admissible to $\mathfrak{hd}(\pi)$, and furthermore, $D_{\mathfrak m}(\pi)\cong D_{\mathfrak m'}(\pi)$ if and only if $\mathfrak r(\mathfrak m, \pi)=\mathfrak r(\mathfrak m', \pi)$.
\end{theorem}

\subsection{More relations to derivatives}

 For $\mathfrak h \in \mathrm{Mult}_{\rho}$ and $\Delta=[a,b]_{\rho} \in \mathrm{Seg}_{\rho}$, set
\[ \varepsilon_{\Delta}(\mathfrak h)=|\left\{ \widetilde{\Delta} \in \mathfrak h[a]: \Delta\subset \widetilde{\Delta} \right\}| . \]


\begin{theorem} \cite[Theorem 9.3]{Ch22+d} \label{thm effect of Steinberg}
Let $\pi \in \mathrm{Irr}_{\rho}$. Let $\Delta\in \mathrm{Seg}_{\rho}$ be admissible to $\pi$. Let $\Delta' \in \mathrm{Seg}_{\rho}$. Suppose either $a(\Delta')>a(\Delta)$; or $\Delta'$ and $\Delta$ are unlinked. Then $\varepsilon_{\Delta'}(D_{\Delta}(\pi))=\varepsilon_{\Delta'}(\mathfrak r(\Delta, \pi))$.
\end{theorem}

\section{Fine Chains} \label{s intersection union derivative closed}

Recall that $\mathcal S(\pi, \tau)$ is defined in Section \ref{ss main results}. We introduce a notion of fine chains in Definition \ref{def fine chain}  in order to give an effective comparison of the effect of two removal processes (Lemma \ref{lem coincide lemma}).

\subsection{A basic idea}

Let $\mathfrak h \in \mathrm{Mult}_{\rho}$. We consider two segments in the form $\Delta=[a,b]_{\rho}$ and $\Delta'=[a+1, b']_{\rho}$ for $b<b'$. In such case, let $\widetilde{\Delta}=\Delta \cup \Delta'=[a,b']_{\rho}$ and $\widetilde{\Delta}'=\Delta\cap \Delta'=[a+1,b]_{\rho}$. Let 
\[   \mathfrak h^*=\mathfrak h-\Upsilon(\Delta, \mathfrak h)+{}^-\Upsilon(\Delta, \mathfrak h), \quad \widetilde{\mathfrak h}=\mathfrak h-\Upsilon(\widetilde{\Delta}, \mathfrak h)+{}^-\Upsilon(\widetilde{\Delta}, \mathfrak h) .\]
 Applying Lemma \ref{lem removal process}, we have
\[  \mathfrak r(\left\{ \Delta', \Delta\right\}, \mathfrak h) =\mathfrak r(\left\{ \Delta', {}^-\Delta\right\}, \mathfrak h^*)
\]
and
\[  \mathfrak r(\left\{ \widetilde{\Delta}', \widetilde{\Delta}\right\}, \mathfrak h)=\mathfrak r(\left\{ \widetilde{\Delta}', {}^-\widetilde{\Delta} \right\}, \widetilde{\mathfrak h}) .
\]
Note that $\left\{ \Delta', {}^-\Delta\right\}=\left\{ \widetilde{\Delta}', {}^-\widetilde{\Delta} \right\}$. The following statements are equivalent:
\begin{enumerate}
\item $\mathfrak r(\left\{ \Delta, \Delta' \right\}, \mathfrak h) = \mathfrak r(\{ \widetilde{\Delta}, \widetilde{\Delta}'\}, \mathfrak h)$;
\item $\mathfrak h^*=\widetilde{\mathfrak h}$;
\item $\Upsilon(\Delta, \mathfrak h)=\Upsilon(\widetilde{\Delta}, \mathfrak h)$.
\end{enumerate}

The general case for the effect of intersection-union processes needs some modifications for the above consideration and we shall focus on the condition (3). In particular, Sections \ref{ss removal segments} and \ref{ss fine chain} will improve Lemma \ref{lem removal process}(1) to do some multiple 'cutting-off' on starting points.

\subsection{Multiple removal of starting points} \label{ss removal segments}

Let $\mathfrak h, \mathfrak n  \in \mathrm{Mult}_{\rho}$. Let $a$ be the smallest integer such that $\mathfrak n[a]\neq \emptyset$. Write $\mathfrak n[a]=\left\{\Delta_1, \ldots, \Delta_k \right\}$.

\begin{enumerate}
\item Suppose $\mathfrak n[a]$ is admissible to $\mathfrak h$. Let $\mathfrak r_i=\mathfrak r(\left\{ \Delta_{i}, \ldots, \Delta_1\right\}, \mathfrak h)$ and $\mathfrak r_0=\mathfrak h$. 
Define 
\[ \mathfrak{fs}(\mathfrak n, \mathfrak h)=\left\{ \Upsilon(\Delta_1,\mathfrak r_0), \ldots, \Upsilon(\Delta_k, \mathfrak r_{k-1})  \right\} .\]
($\mathfrak{fs}$ refers to first segments.)
\item Suppose $\mathfrak n[a]$ is not admissible to $\mathfrak h$. Define $\mathfrak{fs}(\mathfrak n, \mathfrak h)=\emptyset$.
\end{enumerate}

\begin{lemma} \label{lem independent ordering}
The notion $\mathfrak{fs}(\mathfrak n, \mathfrak h)$ is well-defined i.e. independent of an ordering for the segments in $\mathfrak n[a]$.
\end{lemma}

\begin{proof}
One switches a consecutive pair of segments each time, and then applies Lemmas \ref{lem removal process}(4) and \ref{lem removal process}(5).
\end{proof}

Following from definitions, we also have:

\begin{lemma} \label{lem lemma depend on a}
With the notations as above,
\[  \mathfrak{fs}(\mathfrak n, \mathfrak h)=\mathfrak{fs}(\mathfrak n[a], \mathfrak h)=\mathfrak{fs}(\mathfrak n[a], \mathfrak h[a]).
\]
\end{lemma}

We define a truncation of $\mathfrak h$:
\begin{align} \label{eqn truncate multiple points}
 \mathfrak{trr}(\mathfrak n, \mathfrak h) =\mathfrak h-\mathfrak{fs}(\mathfrak n, \mathfrak h) +{}^-(\mathfrak{fs}(\mathfrak n, \mathfrak h)) ,
\end{align}
and a truncation of $\mathfrak n$:
\[ \mathfrak{trd}(\mathfrak n, \mathfrak h)=\mathfrak n-\mathfrak n[a]+{}^-(\mathfrak n[a]). 
\]
Here we use $r$ in $\mathfrak{trr}$ for the derivative 'resultant' multisegment and $d$ for $\mathfrak{trd}$ for 'taking the derivative for the multisegment $\mathfrak n$'.

\begin{example}
Let $\mathfrak h=\left\{ [0,3]_{\rho}, [1,2]_{\rho},[1,4]_{\rho}, [1,5]_{\rho}, [2,3]_{\rho} \right\}$. Let $\mathfrak n=\left\{ [1,3]_{\rho}, [1,5]_{\rho}, [2]_{\rho} \right\}$. 

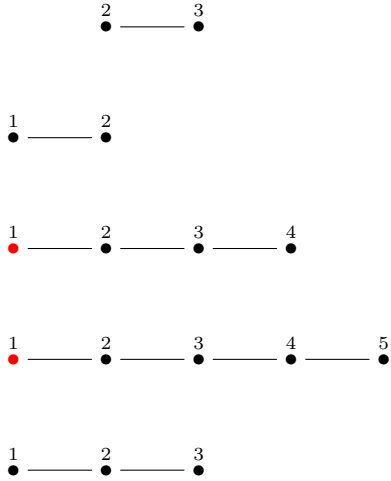
\[ \xymatrix{ &     & \stackrel{2}{ \bullet}  \ar@{-}[r]   &  \stackrel{3}{ \bullet}    &        &    & \\
             & \stackrel{1}{ \bullet}  \ar@{-}[r]  &   \stackrel{2}{ \bullet}     &    &    &     &   \\
              &  \stackrel{1}{ {\color{red} \bullet}}  \ar@{-}[r]   &  \stackrel{2}{ \bullet}   \ar@{-}[r]   &  \stackrel{3}{ \bullet} \ar@{-}[r]    &   \stackrel{4}{\bullet}    &     &     \\     
					   &  \stackrel{1}{ {\color{red} \bullet}}  \ar@{-}[r]    &   \stackrel{2}{ \bullet}  \ar@{-}[r]   &  \stackrel{3}{ \bullet}  \ar@{-}[r]   &  \stackrel{4}{\bullet} \ar@{-}[r]    &   \stackrel{5}{\bullet}   &     \\ 	
		 \stackrel{0}{\bullet}  \ar@{-}[r]												   &  \stackrel{1}{ \bullet}  \ar@{-}[r]    &   \stackrel{2}{ \bullet}  \ar@{-}[r]   &  \stackrel{3}{ \bullet}     &      &    & }
\]
The two red bullets in $\mathfrak h$ are 'truncated' to obtain $\mathfrak{trr}(\mathfrak n, \mathfrak h)=\left\{[0,3]_{\rho}, [1,2]_{\rho}, [2,4]_{\rho}, [2,5]_{\rho}, [2,3]_{\rho} \right\}$. We also have $\mathfrak{fs}(\mathfrak n, \mathfrak h)=\left\{ [1,4]_{\rho}, [1,5]_{\rho} \right\}$ and $\mathfrak{trd}(\mathfrak n, \mathfrak h)=\left\{ [2,3]_{\rho}, [2,5]_{\rho}, [2]_{\rho} \right\}$. 
\end{example}

\begin{lemma} \label{lem multiple truncate} (multiple removal of starting points, c.f. Lemma \ref{lem removal process}(1))
Let $\mathfrak n, \mathfrak h, a$ be as above. 
Then 
\[  \mathfrak r(\mathfrak n, \mathfrak h)=\mathfrak r(\mathfrak{trd}(\mathfrak n, \mathfrak h), \mathfrak{trr}(\mathfrak n, \mathfrak h)) .
\]
\end{lemma}

\begin{proof}

Write $\mathfrak n[a]=\left\{\overline{\Delta}_1, \ldots, \overline{\Delta}_k\right\} $. Suppose that claim holds for positive integers less than $k$, where the basic
case is given by Lemma \ref{lem removal process}(1). Relabeling if necessary, $\overline{\Delta}_1$ is the shortest segment in $\mathfrak n[a]$. Let
\[  \mathfrak h_1^*=\mathfrak h- \overline{\Delta}_1+ {}^-\overline{\Delta}_1  .\]
We observe that:
\begin{align*}
 \mathfrak r(\mathfrak n[a], \mathfrak h)
&= \mathfrak r(\left\{ \overline{\Delta}_2, \ldots, \overline{\Delta}_k \right\}, \mathfrak r(\overline{\Delta}_1, \mathfrak h)) \\
& =\mathfrak r(\left\{ \overline{\Delta}_2, \ldots, \overline{\Delta}_k \right\}, \mathfrak r({}^-\overline{\Delta}_1, \mathfrak h_1^*)) \\
&= \mathfrak r({}^-\overline{\Delta}_1, \mathfrak r(\left\{ \overline{\Delta}_2, \ldots, \overline{\Delta}_k \right\}, \mathfrak h_1^*)) \\
&= \mathfrak r({}^-\overline{\Delta}_1, \mathfrak r(\left\{ {}^-\overline{\Delta}_2, \ldots, {}^-\overline{\Delta}_k \right\}, \mathfrak{trr}(\mathfrak n, \mathfrak h))) \\
&= \mathfrak r({}^-(\mathfrak n[a]), \mathfrak{trr}(\mathfrak n, \mathfrak h)) ,
\end{align*}
where the second equation follows from Lemma \ref{lem removal process}(1), the first, third and last equations follow from Lemma \ref{lem removal process}(5), and the fourth equation follows from the induction hypothesis.

The lemma then follows by applying $\mathfrak r(\mathfrak n-\mathfrak n[a],.)$ on the first and last terms.
\end{proof}

\subsection{Fine chains} \label{ss fine chain}

\begin{definition} \label{def fine chain}
For $c \in \mathbb Z$, we modify the ordering $<_c^a$ on $\mathrm{Mult}^a_{\rho,c} \cup \left\{ \infty\right\}$ as follows. For $\mathfrak p_1$, $\mathfrak p_2$ in $\mathrm{Mult}^a_{\rho,c} \cup \left\{ \infty \right\}$, if $\mathfrak p_1\neq \infty$ and $\mathfrak p_2=\infty$, we also write $\mathfrak p_1 <_c^a \mathfrak p_2$. If $\mathfrak p_1=\mathfrak p_2=\infty$, we write $\mathfrak p_1 \leq_c^a \mathfrak p_2$.

\end{definition}

\begin{definition} \label{def fine chain seq} (Collections of first segments in the removal sequence)
Let $\mathfrak h, \mathfrak n \in \mathrm{Mult}_{\rho}$. Set $\mathfrak n_0=\mathfrak n$ and $\mathfrak h_0=\mathfrak h$. We recursively define:
\[  \mathfrak h_i=\mathfrak{trr}(\mathfrak h_{i-1}, \mathfrak n_{i-1}), \quad \mathfrak n_i=\mathfrak{trd}(\mathfrak h_{i-1}, \mathfrak n_{i-1}) .
\]
The sequence of multisegments
\[  \mathfrak{fs}(\mathfrak n_0, \mathfrak h_0), \mathfrak{fs}(\mathfrak n_1, \mathfrak h_1), \ldots
\]
is called the {\it fine chain} for $(\mathfrak n, \mathfrak h)$. Since we usually fix $\mathfrak h$ and vary $\mathfrak n$ in our use of fine chains, we shall denote the fine chain by $\mathrm{fc}_{\mathfrak h}(\mathfrak n)$. It follows from the definition that 
\[ \mathfrak{fs}(\mathfrak n_i, \mathfrak h_i), \mathfrak{fs}(\mathfrak n_{i+1}, \mathfrak h_{i+1}), \ldots \]
 is also the fine chain for $(\mathfrak n_i, \mathfrak h_i)$. 
\end{definition}

\begin{example}
Let $\mathfrak h=\left\{ [0,4]_{\rho}, [1,5]_{\rho} \right\}$.
\begin{itemize}
\item Let $\mathfrak n=\left\{ [0,1]_{\rho}, [1,2]_{\rho} \right\}$. Then the fine chain for $(\mathfrak n, \mathfrak h)$ takes the form:
\[  \left\{ [0,4]_{\rho} \right\}, \left\{ [1,4]_{\rho}, [1,5]_{\rho} \right\}, \left\{ [2,4]_{\rho} \right\}. 
\]
\item Let $\mathfrak n=\left\{ [0,2]_{\rho}, [1]_{\rho} \right\}$. Then the fine chain for $(\mathfrak n, \mathfrak h)$ is the same as the previous one.
\end{itemize}
\end{example}

\begin{definition}
Let $\mathfrak h \in \mathrm{Mult}_{\rho}$. We say that two fine chains $\mathrm{fc}_{\mathfrak h}(\mathfrak n)$ and $\mathrm{fc}_{\mathfrak h}(\mathfrak n')$ {\it coincide} if 
\begin{enumerate}
\item $\mathfrak r(\mathfrak n, \mathfrak h), \mathfrak r(\mathfrak n', \mathfrak h)\neq \infty$; and
\item the two sequences $\mathrm{fc}_{\mathfrak h}(\mathfrak n)$ and $\mathrm{fc}_{\mathfrak h}(\mathfrak n')$ are equal.
\end{enumerate} 
\end{definition}

\begin{lemma} \label{lem coincide lemma}
Let $\mathfrak h \in \mathrm{Mult}_{\rho}$. Let $\mathfrak n, \mathfrak n' \in \mathrm{Mult}_{\rho}$. Then $\mathfrak r(\mathfrak n, \mathfrak h)=\mathfrak r(\mathfrak n', \mathfrak h) \neq \infty$ if and only if the fine chains $\mathrm{fc}_{\mathfrak h}(\mathfrak n)$ and $\mathrm{fc}_{\mathfrak h}(\mathfrak n')$ coincide.
\end{lemma}

\begin{proof}
We write fine chains $\mathrm{fc}_{\mathfrak h}(\mathfrak n)$ and $\mathrm{fc}_{\mathfrak h}(\mathfrak n')$ with notations in Definition \ref{def fine chain seq} as:
\[  \mathfrak{fs}(\mathfrak n_0, \mathfrak h_0), \mathfrak{fs}(\mathfrak n_1, \mathfrak h_1), \ldots
\]
and
\[  \mathfrak{fs}(\mathfrak n_0', \mathfrak h_0'), \mathfrak{fs}(\mathfrak n_1', \mathfrak h_1'), \ldots
\]
with $\mathfrak n_0=\mathfrak n$, $\mathfrak n_0'=\mathfrak n'$, $\mathfrak h_0=\mathfrak h_0'=\mathfrak h$. 

For the only if direction, Lemma \ref{lem removal process}(2) implies that $\mathfrak h_i=\mathfrak h_i'$ for all $i$. Then, from the construction of $\mathfrak h_i$ and $\mathfrak h_i'$, we have that $\mathfrak{fs}(\mathfrak n_{i-1}, \mathfrak h_{i-1})=\mathfrak{fs}(\mathfrak n_{i-1}', \mathfrak h_{i-1}')$. In other words, the fine chains for $(\mathfrak n, \mathfrak h)$ and $(\mathfrak n', \mathfrak h)$ coincide.

For the if direction, since the fine chains coincide, we must have $\mathfrak h_i=\mathfrak h_i'$ by (\ref{eqn truncate multiple points}). In particular, $\mathfrak r(\mathfrak n, \mathfrak h)=\mathfrak r(\mathfrak n', \mathfrak h)$ as desired.
\end{proof}

\subsection{Fine chain ordering}

Multisegments $\mathfrak n$ and $\mathfrak n'$ are said to be of the {\it same cuspidal support} if $\cup_{\Delta \in \mathfrak n} \Delta=\cup_{\Delta \in \mathfrak n'} \Delta$ (counting multiplicities). 

\begin{definition}
Let $\mathfrak n, \mathfrak n' \in \mathrm{Mult}_{\rho}$ be of the same cuspidal support. Let $\mathfrak h\in \mathrm{Mult}_{\rho}$. Suppose $\mathfrak r(\mathfrak n, \mathfrak h)\neq \infty$ and $\mathfrak r(\mathfrak n', \mathfrak h)\neq \infty$. Write $\mathrm{fc}_{\mathfrak h}(\mathfrak n)$ as $\mathfrak{s}_1, \mathfrak{s}_2, \ldots$ and write $\mathrm{fc}_{\mathfrak h}(\mathfrak n')$ as $\mathfrak{s}_1', \mathfrak{s}_2', \ldots$. Similar to the notations in Definition \ref{def fine chain seq}, set inductively $\mathfrak n_i=\mathfrak n_{i-1}-\mathfrak{s}_i+{}^-\mathfrak{s}_i$ and $\mathfrak n_i'=\mathfrak n'_{i-1}-\mathfrak{s}'_i+{}^-\mathfrak{s}'_i$, where $\mathfrak n_0=\mathfrak n$ and $\mathfrak n_0'=\mathfrak n'$. Let $c_i$ (resp. $c_i'$) be the smallest integer such that $\mathfrak n_i[c_i]\neq \emptyset$ (resp. $\mathfrak n'_i[c_i']\neq \emptyset$).

 We define $\mathfrak n <^{fc} \mathfrak n'$, called the {\it fine chain ordering}, if there exists some $i$ such that for any $j<i$, $\mathfrak{s}_j=\mathfrak{s}_j'$ and  
\[ \mathfrak{s}_i <^a_{c_{i-1}} \mathfrak{s}_i' . \] 
We write $\mathfrak n \leq^{fc} \mathfrak n'$ if either $\mathfrak n <^{fc}\mathfrak n'$ or $\mathrm{fc}_{\mathfrak h}(\mathfrak n)=\mathrm{fc}_{\mathfrak h}(\mathfrak n')$. Note that $\leq^{fc}$ is transitive.
\end{definition}

\section{Closure under intersection-union process} \label{s closure intersect union}

The main result in this section is Theorem \ref{thm closed under zelevinsky}, which proves convex structure of $\mathcal S(\pi, \tau)$. Theorem \ref{thm isomorphic derivatives} transfers the problem on convexity to some explicit combinatorics on the removal process.


\begin{lemma} \label{lem single comparison}
Let $\mathfrak h \in \mathrm{Mult}_{\rho}$. Let $\mathfrak m_1$ be in $\mathrm{Mult}^a_{\rho,c}$. Let $\mathfrak m_2 \in \mathrm{Mult}^a_{\rho, c}$ be obtained from $\mathfrak m_1$ by replacing one segment in $\mathfrak m_1$ with a longer segment of the form $[c,b]_{\rho}$ for some $b \in \mathbb Z$. Then 
\[ \mathfrak{fs}(\mathfrak m_1, \mathfrak h) \leq_c^a \mathfrak{fs}(\mathfrak m_2, \mathfrak h). 
\]
\end{lemma}

\begin{proof}
The only difference between $\mathfrak m_1$ and $\mathfrak m_2$ is in one segment. We can arrange those segments to be the last ones in the process of obtaining $\mathfrak{fs}(\mathfrak m_1, \mathfrak h)$ and $\mathfrak{fs}(\mathfrak m_2, \mathfrak h)$ respectively, by Lemma \ref{lem independent ordering}. Thus the only difference between $\mathfrak{fs}(\mathfrak m_1, \mathfrak h)$ and $\mathfrak{fs}(\mathfrak m_2, \mathfrak h)$ is only one segment. For such ones, the first segment produced for $\mathfrak m_2$ is longer than such produced for $\mathfrak m_1$. The other first segments produced for $\mathfrak m_1$ and $\mathfrak m_2$ coincide. Hence, we then have $\mathfrak{fs}(\mathfrak m_1, \mathfrak h) \leq_c^a \mathfrak{fs}(\mathfrak m_2, \mathfrak h)$. 
\end{proof}

\begin{lemma} \label{lem intersection union fine sequ comp}
Let $\mathfrak h \in \mathrm{Mult}_{\rho}$. Fix $\mathfrak n \in \mathrm{Mult}_{\rho}$. Let $\mathcal N=\mathcal N(\mathfrak n)$ be the set of all multisegments of the same cuspidal support as $\mathfrak n$. Then, for $\mathfrak n', \mathfrak n'' \in \mathcal N$,
\[   \mathfrak n' \leq_Z \mathfrak n'' \quad \Longrightarrow \quad \mathfrak n'' \leq^{fc} \mathfrak n' .
\]

\end{lemma}

\begin{proof}
By the transitivity of $\leq_Z$, we reduce to the case that $\mathfrak n'$ is obtained from $\mathfrak n''$ by an elementary intersection-union operation. Let $\Delta_1$ and $\Delta_2$ be the two linked segments involved in the elementary intersection-union operation. Relabeling if necessary, we write:
\[ \Delta_1=[a_1, b_1]_{\rho}, \quad \Delta_2=[a_2, b_2]_{\rho},
\]
with $a_1<a_2$ and $b_1 < b_2$. 

We again write $\mathrm{fc}_{\mathfrak h}(\mathfrak n')$ as $\mathfrak s_1', \mathfrak s_2', \ldots$ and write $\mathrm{fc}_{\mathfrak h}(\mathfrak n'')$ as $\mathfrak s_1'', \mathfrak s_2'', \ldots$. Similar to notations in Definition \ref{def fine chain seq}, set $\mathfrak n_i'=\mathfrak n_{i-1}'-\mathfrak s'_i+{}^-\mathfrak s_i'$ and $\mathfrak n_i''=\mathfrak n_{i-1}''-\mathfrak s''_i+{}^-\mathfrak s_i''$. Again let $c_i$ be the smallest integer such tht $\mathfrak n_i'[c_i]\neq \emptyset$.  It is straighforward to see from the intersection-union operation that $\mathfrak n_i''[c_i]$ is obtained from $\mathfrak n_i'[c_i]$ by replacing a segment with a longer one (of the form $[c_i,b]_{\rho}$). Thus now Lemma \ref{lem single comparison} implies that $\mathfrak n'' \leq^{fc} \mathfrak n'$.
\end{proof}

\begin{theorem} \label{thm multisegment intersect union}
Let $\mathfrak n, \mathfrak n' \in \mathrm{Mult}_{\rho}$. Suppose $\mathfrak n' \leq_Z \mathfrak n$. Let $\mathfrak n'' \in \mathrm{Mult}_{\rho}$ such that 
\[ \mathfrak n' \leq_Z \mathfrak n'' \leq_Z \mathfrak n.
\]
Then, if $\mathfrak r(\mathfrak n, \mathfrak h)=\mathfrak r(\mathfrak n', \mathfrak h)$, then $\mathfrak r(\mathfrak n, \mathfrak h)=\mathfrak r(\mathfrak n'', \mathfrak h)$.

\end{theorem}

\begin{proof}
If $\mathfrak r(\mathfrak n, \mathfrak h) \neq \mathfrak r(\mathfrak n'', \mathfrak h)$, Lemmas \ref{lem coincide lemma} and \ref{lem intersection union fine sequ comp} imply that $\mathfrak n <^{fc} \mathfrak n''$. By Lemma \ref{lem intersection union fine sequ comp} again, $\mathfrak n<_Z \mathfrak n'$. Now the transitivity of $<^{fc}$ implies that $\mathfrak n <^{fc} \mathfrak n'$. However, Lemma \ref{lem coincide lemma} then implies $\mathfrak r(\mathfrak n', \mathfrak h) \neq \mathfrak r(\mathfrak n, \mathfrak h)$, giving a contradiction.
\end{proof}

We translate the combinatorial statement in Theorem \ref{thm multisegment intersect union} to its representation-theoretic counterpart:

\begin{theorem} \label{thm closed under zelevinsky}
Let $\pi \in \mathrm{Irr}_{\rho}$ and let $\tau$ be a simple quotient of $\pi^{(i)}$. Let $\mathfrak n, \mathfrak n' \in \mathcal S(\pi, \tau)$ with $\mathfrak n' \leq_Z \mathfrak n$. For any $\mathfrak n'' \in \mathrm{Mult}_{\rho}$ such that $\mathfrak n' \leq_Z \mathfrak n'' \leq_Z \mathfrak n$, we have $\mathfrak n'' \in \mathcal S(\pi, \tau)$. 
\end{theorem}

\begin{proof}

This follows from Theorem \ref{thm isomorphic derivatives} and Theorem \ref{thm multisegment intersect union}.
\end{proof}

We also have the following combinatorial consequence:

\begin{corollary}
We use the notations as in Lemma \ref{lem intersection union fine sequ comp}. Let 
\[ \widetilde{\mathcal N}  := \left\{ \mathfrak n \in \mathcal N : \mathfrak r(\mathfrak n, \mathfrak h) \neq \infty \right\} .
\]
We define an equivalence relation $\sim$ on $\widetilde{\mathcal N}$ by: $\mathfrak n \sim \mathfrak n'$ if and only if $\mathfrak r(\mathfrak n, \mathfrak h)=\mathfrak r(\mathfrak n', \mathfrak h)$. Define $\preceq_Z$ on $\widetilde{\mathcal N}/ \sim$ by: for $N, N' \in \widetilde{\mathcal N} / \sim$, write $N \preceq_Z N'$ if there exists $\mathfrak n \in N$ and $\mathfrak n' \in N'$ such that $\mathfrak n \leq_Z \mathfrak n'$. We similarly define the notion $\preceq^{fc}$ on $\widetilde{\mathcal N}$ by replacing $\leq_Z$ with $\leq^{fc}$. Then, the following holds:
\begin{itemize}
\item Both $\preceq_Z$ and $\preceq^{fc}$ define a well-defined poset structure on $\widetilde{\mathcal N}/\sim$. 
\item The identity map on $\widetilde{\mathcal N}/\sim$ induces an order-reversing map between $ (\widetilde{\mathcal N}/\sim, \preceq_Z)$ and $(\widetilde{\mathcal N}/\sim, \preceq^{fc})$. 
\end{itemize}
\end{corollary}

\begin{proof}
For the first bullet, the only non-evident part is the antisymmetry, which indeed follows from Lemmas \ref{lem coincide lemma} and \ref{lem intersection union fine sequ comp}. The second bullet is a direct consequence of Lemma \ref{lem intersection union fine sequ comp}.
\end{proof}

\section{Minimizability} \label{ss minimizable function}

\subsection{Basic example on minimality}

\begin{example} \label{ex non-minimal two segment case}
Let $\mathfrak h=\left\{ [0,5]_{\rho}, [3,8]_{\rho} \right\}$. Let $\mathfrak n=\left\{ [0,3]_{\rho}, [3,4]_{\rho} \right\}$. Then $\mathfrak r([0,3]_{\rho}, \mathfrak h)=\left\{ [4,5]_{\rho}, [3,8]_{\rho} \right\}$ and so $\mathfrak r([3,4]_{\rho}, \mathfrak r([0,3]_{\rho}, \mathfrak h))=\left\{ [4,8]_{\rho}, [5]_{\rho} \right\}$. Note that the segment $[5]_{\rho}$ comes from truncating the segment $[4,5]_{\rho}$ in $\mathfrak r([0,3]_{\rho}, \mathfrak h)$ and the segment $[4,5]_{\rho}$ indeed comes from truncating the segment $[0,5]_{\rho}$. One wonders if one can 'combine' these two effects. Indeed, if one could consider $\mathfrak n'=\left\{ [0,4]_{\rho}, [3]_{\rho} \right\}$, then $\mathfrak r([3]_{\rho}, \mathfrak h)=\left\{ [0,5]_{\rho}, [4,8]_{\rho} \right\}$ and $\mathfrak r([0,4]_{\rho}, \mathfrak r([3]_{\rho}, \mathfrak h))=\left\{ [5]_{\rho}, [4,8]_{\rho} \right\}$. In the last removal process, $[5]_{\rho}$ is obtained directly from truncating $[0,5]_{\rho}$ once.

\end{example}

For convenience, we define a multisegment analogue of $\mathcal S(\pi, \tau)$. For $\mathfrak h, \mathfrak p \in \mathrm{Mult}_{\rho}$,
\[ \mathcal S'( \mathfrak h, \mathfrak p)=\left\{ \mathfrak m \in \mathrm{Mult}_{\rho}: \mathfrak r(\mathfrak m, \mathfrak h)=\mathfrak p \right\} . \]
The above example shows that $\mathfrak n$ is not $\leq_Z$-minimal in $\mathcal S'(\mathfrak h, \mathfrak r(\mathfrak n, \mathfrak h))$. The intuition in Example \ref{ex non-minimal two segment case} will be formulated properly in Section \ref{ss non-overlapping property}, but we shall first deal with more general multisegments (rather than only two segments) below.

\subsection{Local minimizability}

We now define minimizability in Definition \ref{def minimizable function} to show the uniqueness for the $\leq_Z$-minimal element in $\mathcal S(\pi, \tau)$ in Theorem \ref{thm unique module}. 

\begin{definition} \label{def minimizable function}
Let $\mathfrak h \in \mathrm{Mult}_{\rho}$ and let $\mathfrak n \in \mathrm{Mult}_{\rho}$ be admissible to $\mathfrak h$. Let $a$ be the smallest integer such that $\mathfrak n[a]\neq \emptyset$. We say that $(\mathfrak n, \mathfrak h)$ is {\it locally minimizable} if there exists a segment $\overline{\Delta}$ in $\mathfrak n[a+1]$ such that the following holds:
\[  |\left\{ \Delta \in \mathfrak n[a]: \overline{\Delta} \subset \Delta       \right\}| < |\left\{ \Delta \in \mathfrak{fs}(\mathfrak n, \mathfrak h) : \overline{\Delta} \subset \Delta \right\}| .\]
We emphasis that the non-strict inequality $\leq$ always holds. 
\end{definition}

\begin{remark}
We explain more on Definition \ref{def minimizable function}. As suggested from the terminology, the local minimizability $(\mathfrak n, \mathfrak h)$ is to find some $\mathfrak n' <_Z\mathfrak n$ such that $\mathfrak r(\mathfrak n', \mathfrak h)=\mathfrak r(\mathfrak n, \mathfrak h)$. For instance, if all segments $\Delta$ in $\mathfrak n[a]$ satisfy $\overline{\Delta} \subset \Delta$, the removal process guarantees that any $\Delta$ in $\mathfrak{fs}(\mathfrak n, \mathfrak h)$ also satisfies $\overline{\Delta}\subset \Delta$. Hence, the inequality in Definition \ref{def minimizable function} is not satisfied. On the other hand, all segments in $\mathfrak n[a]$ are not linked to $\overline{\Delta}$ and so there is no intersection-union operation for segments in $\mathfrak n[a]$ and $\overline{\Delta}$.

\end{remark}

We have one simple way to check local minimizability by just working on one segment. The proof of the following lemma is straightforward from definitions:

\begin{lemma} \label{lem one segment non local min}
Let $\mathfrak h \in \mathrm{Mult}_{\rho}$ and let $\mathfrak n \in \mathrm{Mult}_{\rho}$ be admissible to $\mathfrak h$. Let $a$ be the smallest integer such that $\mathfrak n[a] \neq \emptyset$. Let $\Delta$ be a segment in $\mathfrak n[a]$. If there exists a segment $\overline{\Delta}$ in $\mathfrak n[a+1]$ such that $\overline{\Delta} \not\subset \Delta$ and $\overline{\Delta} \subset \Upsilon(\Delta, \mathfrak h)$, then $(\mathfrak n, \mathfrak h)$ is locally minimizable.
\end{lemma}




\begin{example}
Let $\mathfrak h=\left\{ [0,1]_{\rho}, [1,4]_{\rho}, [1,5]_{\rho}, [1,6]_{\rho}, [2,5]_{\rho}, [3,4]_{\rho} \right\}$, let $\mathfrak n=\left\{[1,3]_{\rho}, [1,6]_{\rho}, [2,4]_{\rho} \right\}$ and let $\mathfrak n'=\left\{[1,3]_{\rho}, [1,6]_{\rho}, [2,5]_{\rho} \right\}$. 

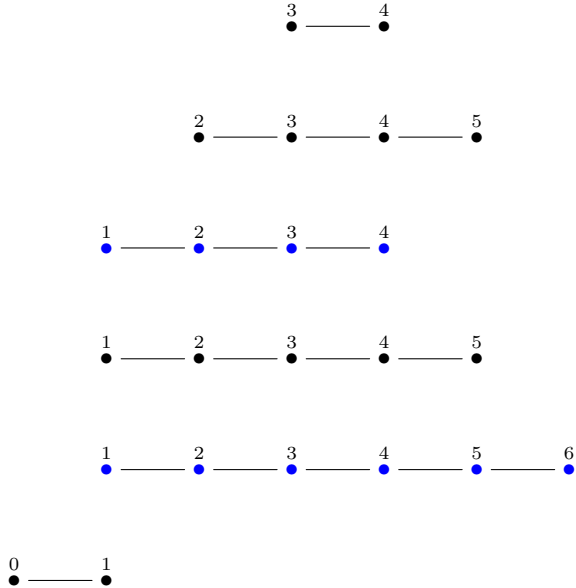
\[ \xymatrix{ & &     &    &  \stackrel{3}{ \bullet} \ar@{-}[r]    &   \stackrel{4}{ \bullet}     &    & \\
    & &     & \stackrel{2}{ \bullet} \ar@{-}[r]    &  \stackrel{3}{ \bullet} \ar@{-}[r]    &   \stackrel{4}{ \bullet}  \ar@{-}[r]   &   \stackrel{5}{ \bullet}   & \\
       &      & \stackrel{1}{{\color{blue} \bullet}}  \ar@{-}[r]  &   \stackrel{2}{{\color{blue} \bullet}} \ar@{-}[r]    &   \stackrel{3}{{\color{blue} \bullet}} \ar@{-}[r]  &  \stackrel{4}{{\color{blue} \bullet}}   &     &   \\
       &       &  \stackrel{1}{  \bullet}  \ar@{-}[r]   &  \stackrel{2}{ \bullet}   \ar@{-}[r]   &  \stackrel{3}{ \bullet} \ar@{-}[r]    &   \stackrel{4}{\bullet}  \ar@{-}[r]  &   \stackrel{5}{\bullet}   &     \\     
			      &	   &  \stackrel{1}{ {\color{blue} \bullet}}  \ar@{-}[r]    &   \stackrel{2}{{\color{blue} \bullet}}  \ar@{-}[r]   &  \stackrel{3}{{\color{blue} \bullet}}  \ar@{-}[r]   &  \stackrel{4}{{\color{blue}\bullet}} \ar@{-}[r]    &   \stackrel{5}{{\color{blue} \bullet}} \ar@{-}[r]  &   \stackrel{6}{{\color{blue} \bullet}}    \\ 	
	&	 \stackrel{0}{\bullet}  \ar@{-}[r]												   &  \stackrel{1}{ \bullet}     &      &       &      &    & }
\]
The blue points represent $\mathfrak{fs}(\mathfrak n, \mathfrak h)$ and $\mathfrak{fs}(\mathfrak n', \mathfrak h)$. Note that 
\[  |\left\{ \Delta \in \mathfrak n[1]: [2,4]_{\rho} \subset \Delta \right\}|=1, \quad |\left\{ \Delta \in \mathfrak{fs}(\mathfrak n, \mathfrak h): [2,4]_{\rho} \subset \Delta \right\}|=2
\]
and so $(\mathfrak n, \mathfrak h)$ is locally minimizable. On the other hand,
\[  |\left\{ \Delta \in \mathfrak n'[1]: [2,5]_{\rho} \subset \Delta \right\}|=1, \quad |\left\{ \Delta \in \mathfrak{fs}(\mathfrak n', \mathfrak h): [2,5]_{\rho} \subset \Delta \right\}|=1 .
\]
Hence $(\mathfrak n', \mathfrak h)$ is not locally minimizable.


\end{example}

\begin{lemma} \label{lem inductive finding intersection}
Let $\mathfrak h \in \mathrm{Mult}_{\rho}$ and let $\mathfrak n \in \mathrm{Mult}_{\rho}$ be admissible to $\mathfrak h$. Let $\mathfrak n'=\mathfrak{trd}(\mathfrak n, \mathfrak h)$ and $\mathfrak h'=\mathfrak{trr}(\mathfrak n, \mathfrak h)$. Let $a$ be the smallest integer such that $\mathfrak n'[a]\neq \emptyset$. Fix some $c>a+1$. Fix a segment $\overline{\Delta}$ of the form $[c,d]_{\rho}$ for some $d$. Suppose
\[  (*) \quad  |\left\{ \Delta \in \mathfrak n'[a+1] :  \overline{\Delta} \subset \Delta \right\}| < |\left\{ \Delta \in \mathfrak{fs}(\mathfrak n', \mathfrak h'): \overline{\Delta} \subset \Delta \right\}| .
\]
\begin{itemize}
\item There exists a segment $\widetilde{\Delta}$ in $\mathfrak n[a]+\mathfrak n[a+1]$ such that 
\[  \mathfrak{fs}(\mathfrak n, \mathfrak h) =\mathfrak{fs}(\widetilde{\mathfrak n}, \mathfrak h), \quad \mathfrak{fs}(\mathfrak n', \mathfrak h')=\mathfrak{fs}(\widetilde{\mathfrak n}', \mathfrak h') ,
\]
where $\widetilde{\mathfrak n}$ is obtained from $\mathfrak n$ by an elementary intersection-union process between $\widetilde{\Delta}$ and $\overline{\Delta}$, and $\widetilde{\mathfrak n}'=\mathfrak{trd}(\widetilde{\mathfrak n}, \mathfrak h)$.
\item Furthermore, if the segment $\widetilde{\Delta}$ cannot be chosen in $\mathfrak n[a+1]$, then 
\[ |\left\{ \Delta \in \mathfrak n[a]: \overline{\Delta} \subset \Delta \right\}|<|\left\{ \Delta \in \mathfrak{fs}(\mathfrak n, \mathfrak h): \overline{\Delta} \subset \Delta \right\}|.  \]
\end{itemize}
\end{lemma}

\begin{proof}
Recall that
\[ {}^-(\mathfrak n[a])= \left\{  {}^-\Delta : \Delta \in \mathfrak n[a], \Delta\neq [a]_{\rho} \right\}. \]
Let $r=|{}^-(\mathfrak n[a])|$ and let $s=|\mathfrak n[a+1]|$. We arrange the segments in $\mathfrak n'[a+1]$ as follows: the first $r$ segments are those in ${}^-(\mathfrak n[a])$, and the remaining segments in $\mathfrak n'[a+1]$ are those in $\mathfrak n[a+1]$. To facilitate discussions, the first $r$ segments in $\mathfrak n'[a+1]$ are labelled as
\[   \Delta_1, \ldots, \Delta_r
 \]
and the remaining segments in $\mathfrak n'[a+1]$ are:
\[  \widetilde{\Delta}_{1}, \ldots, \widetilde{\Delta}_s .
\]
We also set 
\[   \Lambda_i=\Upsilon(\Delta_i, \mathfrak r(\left\{ \Delta_1, \ldots, \Delta_{i-1} \right\}, \mathfrak h') ,
\]
\[  \widetilde{\Lambda}_i=\Upsilon(\widetilde{\Delta}_i, \mathfrak r(\left\{ \widetilde{\Delta}_1, \ldots, \widetilde{\Delta}_{i-1}, \Delta_1, \ldots, \Delta_r\right\}, \mathfrak h')) .
\]

{\bf Case 1:}
\[  |\left\{ \Delta \in \mathfrak n[a]: \overline{\Delta} \subset \Delta \right\}|=|\left\{ \Delta \in \mathfrak{fs}(\mathfrak n, \mathfrak h): \overline{\Delta} \subset \Delta \right\}|.
\]
This condition and the nesting property imply that for the first $r$ segments $\Delta_i$, if $\overline{\Delta} \not\subset \Delta_i$, then 
\[   \overline{\Delta} \not\subset \Lambda_i .
\]
Thus condition (*) implies that there exists a segment $\widetilde{\Delta}_i$ such that $\overline{\Delta}\not\subset \widetilde{\Delta}_i$ and $\overline{\Delta} \subset \widetilde{\Lambda}_i$. Now we do the intersection-union operation on $\overline{\Delta}$ and $\widetilde{\Delta}_i$ to obtain $\widetilde{\mathfrak n}$ from $\mathfrak n$. Then $\mathfrak n[a]=\widetilde{\mathfrak n}[a]$ and so $\mathfrak{fs}(\mathfrak n, \mathfrak h)=\mathfrak{fs}(\widetilde{\mathfrak n}, \mathfrak h)$. And, $\mathfrak n'$ and $\widetilde{\mathfrak n}':=\mathfrak{trd}(\widetilde{\mathfrak n}, \mathfrak h')$ are only differed by $\widetilde{\Delta}_i$ and $\widetilde{\Delta}_i\cup \overline{\Delta}$. However, if we impose the same ordering in computing $\mathfrak{fs}(\widetilde{\mathfrak n}', \mathfrak h')$, it is straightforward to use $\overline{\Delta} \subset \widetilde{\Lambda}_i$ to see that 
\[  \mathfrak{fs}(\mathfrak n', \mathfrak h')=\mathfrak{fs}(\widetilde{\mathfrak n}', \mathfrak h').
\]

{\bf Case 2:}
\[  |\left\{ \Delta \in \mathfrak n[a]: \overline{\Delta} \subset \Delta \right\}| < |\left\{ \Delta \in \mathfrak{fs}(\mathfrak n, \mathfrak h): \overline{\Delta} \subset \Delta \right\}| .
\]

Now, by (*), there exists a segment $\widetilde{\Delta}=\Delta_i$ or $\widetilde{\Delta}_i$ in $\mathfrak n'$ such that $\overline{\Delta} \not\subset \widetilde{\Delta}$ and $\overline{\Delta} \subset \Lambda$, where $\Lambda=\Lambda_i$ or $\widetilde{\Lambda}_i$ according to $\widetilde{\Delta}$. 

If $\widetilde{\Delta}=\widetilde{\Delta}_i$ for some $i$, then the intersection-union operation is done between $\widetilde{\Delta}$ and $\overline{\Delta}$. The argument is similar to Case 1 and we omit the details.

We now consider the case that $\widetilde{\Delta}=\Delta_i$ for some $i$. For convenience, set ${}^+[a+1,e]_{\rho}=[a,e]_{\rho}$ for any $e$. Note that all ${}^+\Delta_k$ ($k=1, \ldots, r$) constitute all the non-singleton segments in $\mathfrak n[a]$. We can use the ordering ${}^+\Delta_1, \ldots, {}^+\Delta_r$ (with other singleton segments at the end) to compute $\mathfrak{fs}(\mathfrak n, \mathfrak h)$; and similarly use that ordering with ${}^+\Delta_i$ replaced by ${}^+\Delta_i\cup \overline{\Delta}$ to compute $\mathfrak{fs}(\widetilde{\mathfrak n}, \mathfrak h)$. 

The only difference in computing $\mathfrak{fs}(\mathfrak n, \mathfrak h)$ and $\mathfrak{fs}(\widetilde{\mathfrak n}, \mathfrak h)$ is to compute the first segments for ${}^+\Delta_i$ and ${}^+\Delta_i\cup \overline{\Delta}$, but we can still guarantee that choices for first segments (for computing $\mathfrak{fs}(\mathfrak n, \mathfrak h)$ and $\mathfrak{fs}(\mathfrak n', \mathfrak h')$) still coincide by using the nesting property of the removal process and the condition $\overline{\Delta}\subset \Lambda$. Hence, $\mathfrak{fs}(\mathfrak n, \mathfrak h)=\mathfrak{fs}(\widetilde{\mathfrak n}, \mathfrak h)$. In more details, we again consider and arrange the segments in $\mathfrak n[a]$ as follows: the first $r$ segments are ${}^+\Delta_1, \ldots, {}^+\Delta_r$ and the remaining segments in $\mathfrak{fs}(\mathfrak n, \mathfrak h)$ are the singleton segments $[a]_{\rho}$. Now, for $i=1, \ldots, r$, let
\[  \Lambda'_i= \Upsilon({}^+\Delta_i, \mathfrak r(\left\{ {}^+\Delta_1, \ldots, {}^+\Delta_{i-1}\right\}, \mathfrak h)) ,   
\] 
where are not singletons (since each ${}^+\Delta_i$ is not a singleton). From definitions, those $\Lambda'_i$ are in $\mathfrak{fs}(\mathfrak n, \mathfrak h)$ and so those ${}^-\Lambda'_i$ are in $\mathfrak h'$ by (\ref{eqn truncate multiple points}). Now, from the minimality in choosing the first segments, we have that:
\[   \Lambda_i \subset \Lambda_i' 
\]
Thus $\Delta_i\cup \overline{\Delta} \subset \Lambda_i$ implies that ${}^+\Delta_i\cup \overline{\Delta} \subset \Lambda_i'$. When one chooses those first segments, the one in $\mathfrak h[a]$ for ${}^+\Delta_i\cup \overline{\Delta}$ will then agree that for ${}^+\Delta_i$.

Computing $\mathfrak{fs}(\mathfrak n', \mathfrak h')=\mathfrak{fs}(\widetilde{\mathfrak n}', \mathfrak h')$ is again similar since the only difference between $\mathfrak n'$ and $\widetilde{\mathfrak n}'$ is $\Delta_i$ and $\Delta_i \cup \overline{\Delta}$. In more details, the first $i-1$ segments for both $\mathfrak n'[a+1]$ and $\widetilde{\mathfrak n}'$ can be arranged as $\Delta_1, \ldots, \Delta_{i-1}$, and so the choices for the first segments in both cases are $\Lambda_1, \ldots, \Lambda_{i-1}$. The $i$-th segment for $\mathfrak n'[a+1]$ can be arranged as $\Delta_i$ while the $i$-th segment for $\widetilde{\mathfrak n}'[a+1]$ can be arranged as $\Delta_i \cup \overline{\Delta} (\subset \Lambda_i)$. Now the next choice for the first segment in the case of $\mathfrak n'[a+1]$ is $\Lambda_i$, and by the minimal choice of $\Lambda_i$ for $\mathfrak n'[a+1]$, the next choice for the first segment in the case of $\widetilde{\mathfrak n}'[a+1]$ must also be $\Lambda_i$. 
\end{proof}

\section{Uniqueness of minimality in $\mathcal S(\pi, \tau)$} \label{s unique minimal}

The terminology of minimizability is suggested by the following lemma:

\begin{lemma} \label{lem maximzable and resultant}
Let $\mathfrak h \in \mathrm{Mult}_{\rho}$ and let $\mathfrak n \in \mathrm{Mult}_{\rho}$ be admissible to $\mathfrak h$. Let the fine chain $\mathrm{fc}_{\mathfrak h}(\mathfrak n)$ take the form
\[ \mathfrak{fs}(\mathfrak n_0, \mathfrak h_0), \mathfrak{fs}(\mathfrak n_1, \mathfrak h_1), \ldots \]
as in Definition \ref{def fine chain seq}. If $(\mathfrak n_j, \mathfrak h_j)$ is not locally minimizable for any $j$, then there is no multisegment $\mathfrak n'$ such that $\mathfrak n' \lneq_Z \mathfrak n$ and $\mathfrak r(\mathfrak n, \mathfrak h)=\mathfrak r(\mathfrak n', \mathfrak h)$. 
\end{lemma}

\begin{proof}
Suppose there is a multisegment $\mathfrak n'$ such that $\mathfrak n' \lneq_Z \mathfrak n$ and $\mathfrak r(\mathfrak n, \mathfrak h)=\mathfrak r(\mathfrak n', \mathfrak h)$. Then, by Theorem \ref{thm closed under zelevinsky}, we may take $\mathfrak n'$ to be obtained from $\mathfrak n$ by an elementary intersection-union process. Let $\widetilde{\Delta}=[\widetilde{a}, \widetilde{b}]_{\rho}, \overline{\Delta}=[\overline{a}, \overline{b}]_{\rho}$ in $\mathfrak n$ be the segments involved in the intersection-union process, and switching labeling if necessary, we may assume that $\widetilde{a} <\overline{a}$.  

We similarly obtain the fine chain 
\[ \mathfrak{fs}(\mathfrak n'_0, \mathfrak h'_0), \mathfrak{fs}(\mathfrak n'_1, \mathfrak h'_1), \ldots, 
\]
for $(\mathfrak n', \mathfrak h)$. We consider $j$ such that $\overline{a}-1$ is the smallest integer $c$ such that $\mathfrak n_j[c]\neq \emptyset$. (Such $j$ exists by using the condition that $\widetilde{\Delta}$ and $\overline{\Delta}$ are linked.) 

Then, in $\mathfrak n_j$, we have a segment $[\overline{a}-1,\widetilde{b}]_{\rho}$ coming by truncating $\widetilde{\Delta}$. If we replace $[\overline{a}-1,\widetilde{b}]_{\rho}$ in $\mathfrak n_j$ by $[\overline{a}-1,\overline{b}]_{\rho}$, this gives $\mathfrak n_j'$. 

Now
\[     |\left\{ \Delta \in \mathfrak n_j[a-1]: \overline{\Delta}\subset \Delta \right\}|          < |\left\{ \Delta \in \mathfrak n_j'[a-1]: \overline{\Delta} \subset \Delta \right\}|  \leq |\left\{ \Delta \in \mathfrak{fs}(\mathfrak n_j', \mathfrak h_j): \widetilde{\Delta} \subset \Delta \right\}|,
\]
where the first strict inequality comes from $[\overline{a}-1, \overline{b}]_{\rho}$. But by Lemma \ref{lem coincide lemma}, two fine chains coincide and in particular $\mathfrak{fs}(\mathfrak n_j, \mathfrak h_j)=\mathfrak{fs}(\mathfrak n'_j, \mathfrak h_j)$. Hence, we now have that $(\mathfrak n_j, \mathfrak h_j)$ is locally minimizable by Definition \ref{def minimizable function} as desired.
\end{proof}

We now prove the converse of Lemma \ref{lem maximzable and resultant}. The main idea is to use Lemma \ref{lem inductive finding intersection} to locate a suitable choice of a segment for the intersection-union process. Since the local minimizability is for the multisegments in the fine chain $\mathrm{fc}_{\mathfrak h}(\mathfrak n)$, one may not be able to immediately find a segment that originally comes from $\mathfrak n$ and so the second bullet of  Lemma \ref{lem inductive finding intersection} allows one to trace back and inductively use Lemma \ref{lem inductive finding intersection} to find suitable segments coming from an original segment in $\mathfrak n$ responsible for the intersection-union operation.

\begin{lemma} \label{lem maxmize intersect}
We keep using the notations as in Lemma \ref{lem maximzable and resultant}. If $(\mathfrak n_j, \mathfrak h_j)$ is locally minimizable for some $j$, then there is a multisegment $\mathfrak n'$ such that $\mathfrak n' \lneq_Z \mathfrak n$ and $\mathfrak r(\mathfrak n, \mathfrak h)=\mathfrak r(\mathfrak n', \mathfrak h')$. 
\end{lemma}

\begin{proof}
We pick any $j$ such that $(\mathfrak n_j, \mathfrak h_j)$ is locally minimizable. Let $a$ be the smallest integer such that $\mathfrak n_j[a]\neq \emptyset$. The below argument is similar if $j=1$ and so we assume $j >1$ for the convenience of using the stated form of Lemma \ref{lem inductive finding intersection}. 

Note that $\mathfrak n_j[a]={}^-(\mathfrak n_{j-1}[a-1])+\mathfrak n[a]$. (Here we have $\mathfrak n_{j-1}[a]=\mathfrak n[a]$.) The local minimizability condition implies that we can use the first bullet of Lemma \ref{lem inductive finding intersection} (set $c=a+1>a-1+1$ in our case) with respect to a certain segment in $\mathfrak n[a+1]$, denoted by $\overline{\Delta}$. Then, that lemma implies that we can find a segment $\widetilde{\Delta}$ in ${}^-(\mathfrak n_{j-1}[a-1])+\mathfrak n[a]$ satisfying the required properties.

The first case is that $\widetilde{\Delta}$ comes from $\mathfrak n[a]$. In this case, let $\widetilde{\mathfrak n}$ be the multisegment obtained from $\mathfrak n$ by the intersection-union operation of the segments $\widetilde{\Delta}$ and $\overline{\Delta}$. Then it is straightforward from the definitions that 
\[  \mathfrak{fs}(\mathfrak n_0, \mathfrak h_0)=\mathfrak{fs}(\widetilde{\mathfrak n}_0, \mathfrak h_0), \ldots, \mathfrak{fs}(\mathfrak n_{j-1}, \mathfrak h_{j-1})=\mathfrak{fs}(\widetilde{\mathfrak n}_{j-1}, \mathfrak h_{j-1}),
\]
where $\mathfrak h_0=\mathfrak h$, $\mathfrak{fs}(\widetilde{\mathfrak n}_k, \mathfrak h_k)$ are the first $j-1$ terms of $\mathrm{fc}_{\mathfrak h}(\widetilde{\mathfrak n})$. However, $\mathfrak{fs}(\mathfrak n_j, \mathfrak h_j)=\mathfrak{fs}(\widetilde{\mathfrak n}_j, \mathfrak h_j)$ is guaranted by Lemma \ref{lem inductive finding intersection}, and so $\mathfrak{trr}(\mathfrak n_j, \mathfrak h_j)=\mathfrak{trr}(\widetilde{\mathfrak n}_j, \mathfrak h_j)$. But then $\mathfrak n_{j+1}=\widetilde{\mathfrak n}_{j+1}$, which follows from that after truncating the points $\nu^c\rho$ (with $c\leq a$) for $\widetilde{\Delta}\cup \overline{\Delta}$ and $\widetilde{\Delta}\cap \overline{\Delta}$, the truncated segments coincide. Now the remaining terms in two fine chains also agree by the definitions. Hence, two fine chains coincide and so $\mathfrak r(\mathfrak n, \mathfrak h)=\mathfrak r(\widetilde{\mathfrak n}, \mathfrak h)$ by Lemma \ref{lem coincide lemma}.

The second case is that $\widetilde{\Delta}$ cannot come from $\mathfrak n[a]$. This case can be obtained using a repeated application of Lemma \ref{lem inductive finding intersection}, followed by Lemma \ref{lem coincide lemma}.
\end{proof}

\begin{proposition} \label{prop unique combin}
Let $\mathfrak h, \mathfrak p \in \mathrm{Mult}_{\rho}$.
 Then there exists a unique minimal element in $\mathcal S'(\mathfrak h, \mathfrak p)$ if $\mathcal S'(\mathfrak h, \mathfrak p)\neq \emptyset$. 
\end{proposition}

\begin{proof}
We first explain the main idea of the proof, which is inductive in nature. For this, for a segment $[a,b]_{\rho} \in \mathrm{Seg}_{\rho}$, define $l_{rel}([a,b]_{\rho})=b-a+1$, and a multisegment $\mathfrak n \in \mathrm{Mult}_{\rho}$, define $l_{rel}(\mathfrak n)=\sum l_{rel}(\Delta)$, where the sum runs for all segments in $\mathfrak n$.  One first picks two minimal multisegments $\mathfrak n$ and $\mathfrak n'$ in $\mathcal S'(\mathfrak h, \mathfrak p)$. One then finds $\prec^L$-minimal segments $\widetilde{\Delta}$ and $\widetilde{\Delta}'$ in $\mathfrak n$ and $\mathfrak n'$ respectively. If $\widetilde{\Delta}=\widetilde{\Delta}'$, then one uses induction on the size of $\mathfrak n$ to argue $\mathfrak n-\widetilde{\Delta}=\mathfrak n'-\widetilde{\Delta}$. If $\widetilde{\Delta} \neq \widetilde{\Delta}'$, then one first reduces to the case that $\widetilde{\Delta} \subsetneq \widetilde{\Delta}'$. Then one applies induction hypothesis on the $l_{rel}(\mathfrak n)$ to show ${}^-\widetilde{\Delta}'$ is also in $\mathfrak n$. Then one shows that $\widetilde{\Delta}$ and $\widetilde{\Delta}'$ in $\mathfrak n$ give rise the local minimizability. We now give more details below.

Let $\mathfrak n, \mathfrak n'$ be two minimal multisegments in $\mathcal S'(\mathfrak h, \mathfrak p)$. Let $a$ be the smallest integer such that $\mathfrak n[a]\neq 0$. Then, by a comparison on cuspidal representations, $a$ is also the smallest integer such that $\mathfrak n'[a]\neq 0$. 

Suppose $\mathfrak n[a]\cap \mathfrak n'[a] \neq \emptyset$. Let $\widetilde{\Delta} \in \mathfrak n[a] \cap \mathfrak n'[a]$. Then we consider $\mathcal S'(\mathfrak r(\widetilde{\Delta}, \mathfrak h), \mathfrak p)$. The minimality for $\mathfrak n$ and $\mathfrak n'$ also guarantees that $\mathfrak n-\widetilde{\Delta}$ and $\mathfrak n'-\widetilde{\Delta}$ are also minimal in $\mathcal S'(\mathfrak r(\widetilde{\Delta}, \mathfrak h), \mathfrak p)$. Thus, by induction on the size of $\mathfrak n$, we have that $\mathfrak n-\widetilde{\Delta}=\mathfrak n'-\widetilde{\Delta}$ and so $\mathfrak n=\mathfrak n'$.

Now suppose $\mathfrak n[a] \cap \mathfrak n'[a] =\emptyset$ to obtain a contradiction. Let $\Delta$ and $\Delta'$ be the shortest segment in $\mathfrak n[a]$ and $\mathfrak n'[a]$ respectively. Switching labeling if necessary, we may assume that $\Delta \subsetneq \Delta'$. Then, by Lemma \ref{lem multiple truncate},  
\[  \mathfrak r( \mathfrak n, \mathfrak h)=\mathfrak r(\mathfrak{trd}(\mathfrak n, \mathfrak h), \mathfrak{trr}(\mathfrak n, \mathfrak h)), \quad \mathfrak r(\mathfrak n', \mathfrak h)=\mathfrak r(\mathfrak{trd}(\mathfrak n', \mathfrak h), \mathfrak{trr}(\mathfrak n', \mathfrak h)) .
\]
By Lemma \ref{lem coincide lemma}, 
\[ \mathfrak{trr}(\mathfrak n, \mathfrak h)=\mathfrak{trr}(\mathfrak n', \mathfrak h) .
\]

By Lemma \ref{lem maxmize intersect}, $(\mathfrak{trr}(\mathfrak n, \mathfrak h), \mathfrak{trd}(\mathfrak n, \mathfrak h))$ and the terms from the fine chains are not locally minimizable. Similarly, this also holds for  $(\mathfrak{trr}(\mathfrak n', \mathfrak h), \mathfrak{trd}(\mathfrak n, \mathfrak h))$. However, Lemma \ref{lem maximzable and resultant} implies that both $\mathfrak{trd}(\mathfrak n, \mathfrak h)$ and $\mathfrak{trd}(\mathfrak n', \mathfrak h)$ are minimal in $\mathcal S'(\mathfrak{trr}(\mathfrak n, \mathfrak h) ,\mathfrak p)=\mathcal S'(\mathfrak{trr}(\mathfrak n', \mathfrak h), \mathfrak p)$. Hence, by induction on $l_{rel}(\mathfrak n)$,
\[  \mathfrak{trd}(\mathfrak n, \mathfrak h)=\mathfrak{trd}(\mathfrak n', \mathfrak h) .
\]
But then, the disjointness assumption implies that ${}^-\Delta' \in \mathfrak n$. But $ {}^-\Delta' \not\subset \Delta$ and ${}^-\Delta' \subset \Upsilon(\Delta, \mathfrak h)$. By Lemma \ref{lem one segment non local min}, this implies 
\[   |\left\{ \widetilde{\Delta} \in \mathfrak n[a]: {}^-\Delta'\subset \widetilde{\Delta}  \right\}|< |\left\{ \widetilde{\Delta} \in \mathfrak{fs}(\mathfrak n, \mathfrak h):  {}^-\Delta'\subset \widetilde{\Delta} \right\}| .
\]
Hence, $(\mathfrak n, \mathfrak h)$ is locally minimizable. This contradicts to Lemma \ref{lem maxmize intersect}.
 \end{proof}

\begin{theorem} \label{thm unique module}
Let $\pi \in \mathrm{Irr}_{\rho}$ and let $\tau$ be a simple quotient of $\pi^{(i)}$ for some $i$. Then $\mathcal S(\pi, \tau)$ has a unique minimal element if $\mathcal S(\pi, \tau)\neq \emptyset$. Here the minimality is with respect to $\leq_Z$. 
\end{theorem}

\begin{proof}
This follows from Proposition \ref{prop unique combin} and Theorem \ref{thm isomorphic derivatives}.
\end{proof}

\section{Examples of minimality} \label{s example minimal}

\subsection{Minimality for the highest derivative multisegment} \label{ss proof of thm}

For $\pi \in \mathrm{Irr}_{\rho}$, let $\pi^-$ be the Bernstein-Zelevinsky highest derivative of $\pi$ in the sense of \cite[Section 3.5]{BZ77}. We shall not need the explicit definition, but we only use \cite[Theorem 1.3]{Ch22+d}.

\begin{theorem} \label{thm minimal highest derivative}
Let $\pi \in \mathrm{Irr}_{\rho}$. Then $\mathfrak{hd}(\pi)$ is minimal in $\mathcal S(\pi, \pi^-)$.
\end{theorem}

\begin{proof}
It is shown in \cite[Theorem 1.3]{Ch22+d} that $D_{\mathfrak{hd}(\pi)}(\pi)\cong \pi^-$. Theorem \ref{thm closed under zelevinsky} reduces to show that if $\mathfrak n$ is a multisegment obtained by an elementary intersection-union process from $\mathfrak{hd}(\pi)$, then $D_{\mathfrak n}(\pi)=0$. 

Let $\Delta_1=[a_1,b_1]_{\rho}, \Delta_2=[a_2,b_2]_{\rho}$ be two linked segments in $\mathfrak{hd}(\pi)$. Relabeling if necessary, we assume that $a_1<a_2$. Define
\[  \mathfrak n=\mathfrak{hd}(\pi)-\left\{ \Delta_1, \Delta_2 \right\}+\Delta_1\cup \Delta_2+ \Delta_1\cap \Delta_2 . \]
Then, $\mathfrak n[e]=\mathfrak{hd}(\pi)[e]$ for any $e <a_1$ and $\mathfrak n[a_1]\not\leq_{a_1}\mathfrak{hd}(\pi)[a_1]$. 

Let $c$ be the smallest integer such that $\mathfrak{hd}(\pi)[c]\neq 0$. By Theorems \ref{thm isomorphic derivatives} and \ref{thm effect of Steinberg}, 
\[  D_{\mathfrak n[a_1]} \ldots D_{\mathfrak n[c]}(\pi)=D_{\mathfrak n[a_1]}\circ D_{\mathfrak{hd}(\pi)[a_1-1]}\circ \ldots \circ D_{\mathfrak{hd}(\pi)[c]}(\pi) =0 ,
\]
and so $\mathfrak n \notin \mathcal S(\pi, \pi^-)$. 
\end{proof}

\subsection{Minimal multisegment for the generic case}

\begin{proposition} \label{prop unique generic}
Let $\pi \in \mathrm{Irr}_{\rho}$ be generic. Let $\tau$ be a (generic) simple quotient of $\pi^{(i)}$ for some $i$. Then the minimal multisegment in $\mathcal S(\pi, \tau)$ is generic i.e. any two segments in the minimal multisegment are unlinked.
\end{proposition}

One may prove the above proposition by some analysis of derivative resultant multisegments. We shall give another proof using the following lemma:

\begin{lemma} \label{lem jacquet generic quotient}
Let $\pi \in \mathrm{Irr}_{\rho}(G_n)$ be generic. For any $i$, and for any irreducible submodule $\tau_1 \boxtimes \tau_2$ of $\pi_{N_i}$ as $G_{n-i}\times G_i$-representation, both $\tau_1$ and $\tau_2$ are generic.
\end{lemma}

\begin{proof}
Recall that for a generic $\pi$, 
\[   \pi \cong \mathrm{St}(\Delta_1) \times \ldots \times \mathrm{St}(\Delta_r) 
\]
for mutually unlinked segments $\Delta_1, \ldots, \Delta_r$. 

Now, with a suitable arrangment on the orderings of the segments, one may argue as in \cite[Corollary 2.6]{Ch21} to have that a simple quotient of $\pi_{N_i}$ takes the form $\tau \boxtimes \omega$ for some generic $\tau \in \mathrm{Irr}(G_{n-i})$. Hence it remains to show $\omega$ is also generic. We consider
\[   \pi_{N_i} \twoheadrightarrow \tau \boxtimes \omega 
\]
and taking the twisted Jacquet functor on the $G_{n-i}$-parts yields that
\[  {}^{(n-i)}\pi \twoheadrightarrow \omega .
\]
Now using \cite[Corollary 2.6]{Ch21} for left derivatives, we have that $\omega$ is also generic as desired.
\end{proof}

\noindent
{\it Proof of Proposition \ref{prop unique generic}.} Let $\pi \in \mathrm{Irr}_{\rho}$ be generic and let $\tau$ be a simple (generic) quotient of $\pi^{(i)}$. Then, $\pi_{N_i}$ has a simple quotient of the form $\tau \boxtimes \omega$ for some $\omega \in \mathrm{Irr}_{\rho}(G_i)$. By Lemma \ref{lem jacquet generic quotient}, $\omega$ is also  generic and hence $\omega \cong \mathrm{St}(\Delta_1)\times \ldots \times \mathrm{St}(\Delta_k)$ for some mutually unlinked segments $\Delta_1, \ldots, \Delta_k$. Now, $\pi$ is the unique submodule of $\tau \times\mathrm{St}(\Delta_1)\times \ldots \times \mathrm{St}(\Delta_k)$. By a standard argument, we have that:
\[ D_{\Delta_1}\circ \ldots \circ  D_{\Delta_k}(\pi) \cong \tau .
\]
Hence, $\left\{ \Delta_1, \ldots, \Delta_k \right\} \in \mathcal S(\pi, \tau)$. The minimality of $\left\{\Delta_1, \ldots, \Delta_k\right\}$ is automatic since any generic multisegment is minimal in $\mathrm{Mult}_{\rho}$ with respect to $\leq_Z$. Now the statement follows from the uniqueness in Theorem \ref{thm unique module}. \qed

\section{Non-uniqueness of maximal elements in $\mathcal S(\pi, \tau)$} \label{no unique max element}

\subsection{Highest derivative multisegments}

Let $\pi \in \mathrm{Irr}_{\rho}$. Then $\mathcal S(\pi, \pi^-)$ contains a unique maximal multisegment, and such multisegment has all segments to be singletons. Combining with Theorem \ref{thm closed under zelevinsky}, one can describe all multisegments in $\mathcal S(\pi, \pi^-)$.

\subsection{Failure of uniqueness of maximality}
As mentioned in \cite{Ch22+d}, in general, derivatives of cuspidal representations are not enough for constructing all simple quotients of Bernstein-Zelevinsky derivatives, and the set $\mathcal S(\pi, \tau)$ may contain some multisegments whose segments are not all singletons. We give an example to show that in general, there is no uniqueness for $\leq_Z$-maximal elements in $\mathcal S(\pi, \tau)$. 

Let 
\[ \mathfrak h=\left\{ [0,3]_{\rho}, [0,1]_{\rho}, [1,2]_{\rho}, [1,2]_{\rho}, [2]_{\rho}, [3]_{\rho} \right\} .
\]
Let $\mathfrak n=\left\{ [0,3]_{\rho}, [1,2]_{\rho} \right\}$. Then 
\[ \mathfrak r:=\mathfrak r(\mathfrak n, \mathfrak h)=\left\{ [0,1]_{\rho}, [1,2]_{\rho}, [2]_{\rho}, [3]_{\rho} \right\} =\mathfrak h-\mathfrak n.
\]

We claim that 
\[ \mathcal S'(\mathfrak h, \mathfrak r)=\left\{ \mathfrak n, \left\{ [0,3]_{\rho}, [1]_{\rho}, [2]_{\rho} \right\}, \left\{ [0,2]_{\rho}, [1,3]_{\rho} \right\} \right\} .
\]
It is direct to check that the three elements are in $\mathcal S(\mathfrak h, \mathfrak r)$, and the last two elements are both maximal.

To see that there are no more elements, we first observe that any multisegment $\mathfrak n'$ in $\mathcal S'(\mathfrak h, \mathfrak r)$ has only one segment $\widetilde{\Delta}$ with $a(\widetilde{\Delta})=\nu^0$. By considering the first segment in the removal sequence $\mathfrak r(\widetilde{\Delta}, \mathfrak h)$, we note that $[0]_{\rho}, [0,1]_{\rho} \notin \mathcal S'(\mathfrak h, \mathfrak r)$. In other words, $[0,2]_{\rho} $ or $[0,3]_{\rho} $ in $\mathcal S'(\mathfrak h, \mathfrak r)$. It remains to check that the following three elements:
\[  \left\{ [0,2]_{\rho}, [1]_{\rho}, [2]_{\rho}, [3]_{\rho} \right\}, \left\{ [0,2]_{\rho} ,[1,2]_{\rho} ,[3]_{\rho}  \right\}, \left\{ [0,2]_{\rho} , [1]_{\rho} , [2,3]_{\rho}  \right\}
\]
are not in $\mathcal S'(\mathfrak h, \mathfrak r)$, which is straightforward.

To show that the maximality for $\mathcal S(\pi, \tau)$ is not unique in general, one still needs to ask whether there exists $\pi \in \mathrm{Irr}_{\rho}$ such that $\mathfrak{hd}(\pi)=\mathfrak h$. This is indeed the case and we shall postpone proving this in the sequel \cite{Ch22+e}, where we shall need such fact in a more substantial way.

\section{Minimality for two segment case} \label{s dagger property}

In this section, we study the minimality for two segment cases.

\subsection{Non-overlapping property} \label{ss non-overlapping property}
\begin{definition} \label{def nonoverlapping property}
Let $\mathfrak h \in \mathrm{Mult}_{\rho}$. Let $\Delta \in \mathrm{Seg}_{\rho}$ be admissible to $\mathfrak h$. Let $\Delta' \in \mathrm{Seg}_{\rho}$ linked to $\Delta$ with $\Delta'>\Delta$. We say that the triple $(\Delta, \Delta', \mathfrak h)$ satisfies the {\it  non-overlapping property} if for the shortest segment $\overline{\Delta}$ in the removal sequence for $(\Delta, \mathfrak h)$ that contains $\nu^{-1}a(\Delta')$,  we have $\Delta'\not\subset \overline{\Delta}$. (We remark that for later applications, we do not impose the condition that $\Delta'$ is admissible to $\mathfrak h$.)

\end{definition}

We now briefly explain the connection to local minimizability. We use notations in Definition \ref{def nonoverlapping property} and simply consider $a(\Delta)\cong \nu^{-1}a(\Delta')$ and set $a(\Delta')=\nu^{a+1}\rho$ for some $a \in \mathbb Z$. We let $\mathfrak n=\left\{ \Delta, \Delta'\right\}$ and in such case $\mathfrak{fs}(\mathfrak n, \mathfrak h)=\left\{ \Upsilon(\Delta, \mathfrak h)\right\}$. Now, note that $(\Delta, \Delta', \mathfrak h)$ satisfies the non-overlapping property if and only if 
\[  | \left\{ \widetilde{\Delta} \in \mathfrak{fs}(\mathfrak n, \mathfrak h) : \Delta' \subset \widetilde{\Delta} \right\}| = 0  \]
On the other hand, $\mathfrak n[a]=\left\{ \Delta \right\}$ and so 
\[   |\left\{ \widetilde{\Delta} \in  \mathfrak n[a] : \Delta' \subset \widetilde{\Delta} \right\}|=0 .
\]
Combining above, we have that $(\Delta, \Delta', \mathfrak h)$ satisfies the non-overlapping property if and only if $(\mathfrak n, \mathfrak h)$ is not locally minimizable.

\begin{example}
\begin{enumerate}
\item Let $\mathfrak h=\left\{ [0,7]_{\rho}, [3,6]_{\rho}, [6,10]_{\rho}  \right\}$. Let $\Delta=[0,5]_{\rho}$ and let $\Delta'=[6,7]_{\rho}$. Then $(\Delta, \Delta', \mathfrak h)$ satisfies the non-overlapping property.
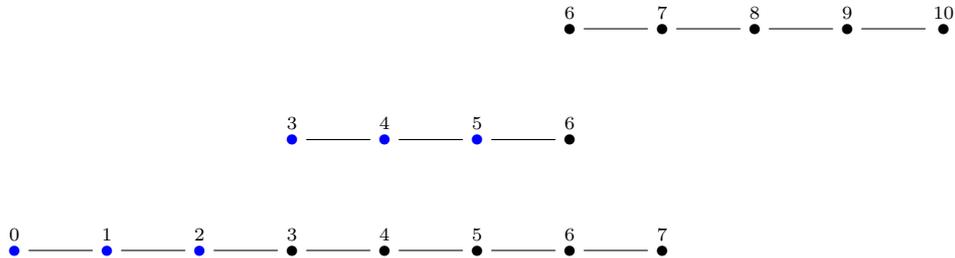
\[ \xymatrix{     &    &    &    &   &    & \stackrel{6}{ \bullet} \ar@{-}[r]  & \stackrel{7}{ \bullet} \ar@{-}[r]  & \stackrel{8}{\bullet} \ar@{-}[r]  & \stackrel{9}{\bullet} \ar@{-}[r]  & \stackrel{10}{\bullet}    \\
  &  &    & \stackrel{3}{ {\color{blue} \bullet}} \ar@{-}[r]  & \stackrel{4}{{\color{blue} \bullet}} \ar@{-}[r]  & \stackrel{5}{{\color{blue} \bullet}} \ar@{-}[r]  & \stackrel{6}{\bullet}   &   &    &    &       \\
 \stackrel{0}{{\color{blue} \bullet}} \ar@{-}[r] & \stackrel{1}{{\color{blue}{\bullet}}} \ar@{-}[r]  & \stackrel{2}{{\color{blue} \bullet}} \ar@{-}[r]  & \stackrel{3}{\bullet} \ar@{-}[r]  & \stackrel{4}{\bullet} \ar@{-}[r]  & \stackrel{5}{\bullet} \ar@{-}[r]  & \stackrel{6}{\bullet} \ar@{-}[r]  & \stackrel{7}{\bullet}    &    &   &   }
\]
The blue points are those points removed by applying $\mathfrak r(\Delta,.)$. Note that the shortest segment in the removal sequence containing $[5]_{\rho}$ is $[3,6]_{\rho}$, which does not contain $[6,7]_{\rho}$. 
\item Let $\mathfrak h=\left\{ [0,8]_{\rho}, [3,6]_{\rho}, [6,10]_{\rho}  \right\}$. Let $\Delta=[0,7]_{\rho}$ and let $\Delta'=[6,8]_{\rho}$. Then $(\Delta, \Delta', \mathfrak h)$ does not satisfy the non-overlapping property. The graph for carrying out the removal sequence looks like:
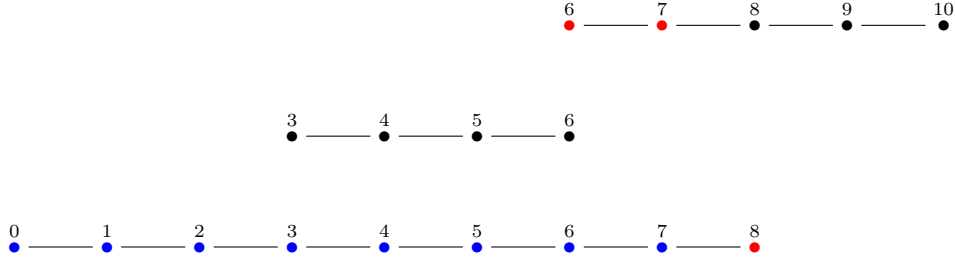
\[ \xymatrix{     &    &    &    &   &    & \stackrel{6}{{\color{red} \bullet}} \ar@{-}[r]  & \stackrel{7}{{\color{red} \bullet}} \ar@{-}[r]  & \stackrel{8}{\bullet} \ar@{-}[r]  & \stackrel{9}{\bullet} \ar@{-}[r]  & \stackrel{10}{\bullet}    \\
  &  &    & \stackrel{3}{  \bullet} \ar@{-}[r]  & \stackrel{4}{ \bullet} \ar@{-}[r]  & \stackrel{5}{ \bullet} \ar@{-}[r]  & \stackrel{6}{\bullet}   &   &    &    &       \\
 \stackrel{0}{{\color{blue} \bullet}} \ar@{-}[r] & \stackrel{1}{{\color{blue}{\bullet}}} \ar@{-}[r]  & \stackrel{2}{{\color{blue} \bullet}} \ar@{-}[r]  & \stackrel{3}{{\color{blue}{\bullet}}} \ar@{-}[r]  & \stackrel{4}{{\color{blue} \bullet}} \ar@{-}[r]  & \stackrel{5}{{\color{blue} \bullet}} \ar@{-}[r]  & \stackrel{6}{{\color{blue} \bullet}} \ar@{-}[r]  & \stackrel{7}{{\color{blue} \bullet}} \ar@{-}[r]   &  \stackrel{8}{{\color{red} \bullet}}  &   &   }
\]
In the graph above, the segment $[0,8]_{\rho}$ contains $[5]_{\rho}$, and $[6,8]_{\rho} \subset [0,8]_{\rho}$. The blue points represent points 'deleted' under the removal process $\mathfrak r([0,7]_{\rho}, \mathfrak h)$ and the red points represent points 'deleted' under the subsequent removal process $\mathfrak r([6,8]_{\rho}, \mathfrak r([0,7]_{\rho}, \mathfrak h))$.
\end{enumerate}

\end{example}

\begin{lemma} \label{lem construct less minimal element}
Let $\mathfrak h \in \mathrm{Mult}_{\rho}$. Let $\Delta \in \mathrm{Seg}_{\rho}$ be admissible to $\mathfrak h$. Let $\Delta' \in \mathrm{Seg}_{\rho}$ be admissible to $\mathfrak r(\Delta, \mathfrak h)$. Suppose $\Delta'$ is linked to $\Delta$ with $\Delta'>\Delta$. Then $(\Delta, \Delta', \mathfrak h)$ does not satisfy the non-overlapping property  if and only if
\[  \mathfrak r(\left\{ \Delta \cap \Delta', \Delta\cup \Delta' \right\}, \mathfrak h)=\mathfrak r(\left\{ \Delta, \Delta' \right\}, \mathfrak h) .
\]
\end{lemma}

\begin{proof}
Suppose $\mathfrak r(\left\{ \Delta \cap \Delta', \Delta \cup \Delta' \right\}, \mathfrak h)\neq \mathfrak r(\left\{ \Delta, \Delta' \right\}, \mathfrak h)$. Lemma \ref{lem removal process}(1) and the nesting property in the removal process reduce to the case that $a(\Delta)\cong \nu^{-1}a(\Delta')$. Now, showing not satisfying non-overlapping property is simply a reformulation of local minimizability by Lemma \ref{lem maxmize intersect}.

Suppose $\mathfrak r(\left\{ \Delta \cap \Delta', \Delta \cup \Delta' \right\}, \mathfrak h)=\mathfrak r(\left\{ \Delta, \Delta' \right\}, \mathfrak h)$. By Lemma \ref{lem removal process}(1), it again reduces to $a(\Delta)\cong \nu^{-1}a(\Delta')$. It then follows from Lemma \ref{lem maximzable and resultant} that $(\left\{ \Delta, \Delta' \right\}, \mathfrak h)$ is locally minimizable and so $(\Delta, \Delta', \mathfrak h)$ does not satisfy the non-overlapping property.
\end{proof}

\subsection{Intermediate segment property}

\begin{definition}
Let $\mathfrak h \in \mathrm{Mult}_{\rho}$. Let $\Delta \in \mathrm{Seg}_{\rho}$ be admissible to $\mathfrak h$. Let $\Delta' \in \mathrm{Seg}_{\rho}$ linked to $\Delta$ with $\Delta' >\Delta$. We say that the triple $(\Delta, \Delta', \mathfrak h)$ satisfies the {\it intermediate segment property} if there exists a segment $\widetilde{\Delta}$ in $\mathfrak h$ such that 
\begin{equation} \label{eqn intermediate segment prop}
 a(\Delta)\leq a(\widetilde{\Delta}) <a(\Delta'), \mbox{ and } b(\Delta) \leq b(\widetilde{\Delta}) <b(\Delta') . 
\end{equation}
\end{definition}

\subsection{Criteria in terms of $\eta$-invariants} \label{ss rel to epsilon}

Let $\mathfrak h \in \mathrm{Mult}_{\rho}$. For a segment $\Delta=[a,b]_{\rho}$ admissible to $\mathfrak h$, note that, by Theorem \ref{thm effect of Steinberg}, $\varepsilon_{\Delta}(\mathfrak{hd}(\pi))=\varepsilon_{\Delta}(\pi)$. Let 
\begin{eqnarray} \label{eqn segments}
  \eta_{\Delta}(\mathfrak h)=(\varepsilon_{[a,b]_{\rho}}(\mathfrak h), \varepsilon_{[a+1,b]_{\rho}}(\mathfrak h), \ldots, \varepsilon_{[b,b]_{\rho}}(\mathfrak h)) .
\end{eqnarray}
The $\eta$-invariant defined above plays an important role in defining a notion of {\it generalized GGP relevant pairs} in \cite{Ch22+b}. When $a=b$, it is sometimes called the highest $\rho$-derivative.

\begin{proposition} \label{prop dagger eta condition} 
Let $\mathfrak h\in \mathrm{Mult}_{\rho}$. Let $\Delta \in \mathrm{Seg}_{\rho}$ be admissible to $\mathfrak h$. Let $\Delta' \in \mathrm{Seg}_{\rho}$ be linked to $\Delta$ with $\Delta'>\Delta$. Then the following conditions are equivalent:
\begin{enumerate}
\item The triple $(\Delta, \Delta', \mathfrak h)$ satisfies the non-overlapping property.
\item $\eta_{\Delta'}(\mathfrak h) = \eta_{\Delta'}(\mathfrak r(\Delta, \mathfrak h))$.
\item The triple $(\Delta, \Delta', \mathfrak h)$ satisfies the intermediate segment property.
\end{enumerate}

\end{proposition}

\begin{proof}
We first prove (3) implies (2). Suppose (3) holds. We denote by 
\[  \Delta_1, \ldots, \Delta_r 
\]
the removal sequence for $(\Delta, \mathfrak h)$. Using (3) and (4) of the removal process in Definition \ref{def removal process}, those $\Delta_1, \ldots, \Delta_r$ in $\mathfrak h$ are replaced by their respective truncations, denoted by
\[  \Delta_1^{tr}, \ldots, \Delta_r^{tr} .\]
By using the intermediate segment property and the minimality condition in the removal process, there exists a segment of the form (\ref{eqn intermediate segment prop}) in the removal sequence for $(\Delta, \mathfrak h)$. Let $i^*$ be the smallest index such that $\Delta_{i^*}$ satisfies (\ref{eqn intermediate segment prop}). Note that, by considering $a(\Delta_j)$,
\[  \Delta_1, \ldots, \Delta_{i^*-1}
\]
do not contribute to $\eta_{\Delta'}(\mathfrak h)$ by definitions. By the definition of truncation and (\ref{eqn intermediate segment prop}) for $\Delta_{i^*}$, we have that $\Delta_1^{tr}, \ldots, \Delta_{i^*-1}^{tr}$ also do not contribute to $\eta_{\Delta'}(\mathfrak r(\Delta,\mathfrak h))$. From the choice of $\Delta_{i^*}$ and the nesting property, we also have that, by considering $b(\Delta_j)$,
\[ \Delta_{i^*}, \ldots, \Delta_r
\]
do not contribute to $\eta_{\Delta'}(\mathfrak h)$, and similarly, $\Delta_{i^*}^{tr}, \ldots, \Delta_r^{tr}$ do not contribute to $\eta_{\Delta'}(\mathfrak r(\Delta, \mathfrak h))$. Thus, we have that 
\[  \eta_{\Delta'}(\mathfrak h)=\eta_{\Delta'}(\mathfrak r(\Delta, \mathfrak h)) .
\]

We now prove (2) implies (1). Again write $\Delta'=[a',b']_{\rho}$. Suppose $(\Delta, \Delta', \mathfrak h)$ does not satisfy the nonoverlapping property. Again, denote by 
\[  \Delta_1, \ldots, \Delta_r
\]
the removal sequence for $(\Delta, \mathfrak h)$. Let $\Delta_l$ be the shortest segment in the removal sequence containing $\nu^{-1}a(\Delta')$. Note that
\[  \Delta_1, \ldots, \Delta_{l-1}, \Delta_l
\]
do not contribute to $\eta_{\Delta'}(\mathfrak h)$ (by considering $a(\Delta_i)$) and similarly, 
\[  \Delta_1^{tr}, \ldots, \Delta_{l-1}^{tr}
\]
do not contribute to $\eta_{\Delta'}(\mathfrak r(\Delta, \mathfrak h))$. However, $\Delta_l^{tr}$ contributes to $\eta_{\Delta'}(\mathfrak r(\Delta, \mathfrak h))$. This causes a difference of $1$ in the coordinate $\varepsilon_{\Delta'}$ for $\eta_{\Delta'}(\mathfrak h)$ and $\eta_{\Delta'}(\mathfrak r(\Delta, \mathfrak h))$. 

It remains to see the following claim: \\
{\it Claim:} For $k >l$, $\Delta_k$ contributes to $\eta_{\Delta'}(\mathfrak h)$ if and only if $\Delta_k^{tr}$ contributes to $\eta_{\Delta'}(\mathfrak r(\Delta, \mathfrak h))$.  \\

\noindent
{\it Proof of claim:}  If $\Delta_k$ does not contribute to $\eta_{\Delta'}(\mathfrak h)$, then $b(\Delta_k)<b(\Delta')$ and so $b(\Delta_k^{tr})<b(\Delta')$ (or $\Delta_k^{tr}$ is dropped or a empty set). This implies that $\Delta_k^{tr}$ does not contribute to $\eta_{\Delta'}(\mathfrak r(\Delta, \mathfrak h))$. 

On the other hand, if $\Delta_k$ contributes to $\eta_{\Delta'}(\mathfrak h)$, then $b(\Delta_k) \geq b(\Delta')$. Note that $\Delta_k^{tr}$ is non-empty by using $\Delta <\Delta'$. Thus we also have $b(\Delta_k^{tr}) \geq b(\Delta')$. We also have that $\Delta_k^{tr}$ contributes to $\eta_{\Delta'}(\mathfrak r(\Delta, \mathfrak h))$. This completes proving the claim.\\

Note that (1) $\Rightarrow$ (3) follows from the segment involved in the definition of overlapping property.

\end{proof}

\begin{example}
Let $\mathfrak h=\left\{ [0,5]_{\rho}, [3,8]_{\rho} \right\}$. 
\begin{itemize}
\item Let $\Delta=[0,3]_{\rho}$ and let $\Delta'=[3,6]_{\rho}$. In such case, $\eta_{\Delta'}(\mathfrak h)=(1,0,0,0)$. Similarly, $\eta_{\Delta'}(\mathfrak r(\Delta, \mathfrak h))=\eta_{\Delta'}(\left\{ [3,8]_{\rho}, [4,5]_{\rho} \right\})=(1,0,0,0)$. 
\item Let $\Delta=[0,3]_{\rho}$ and let $\Delta'=[3,4]_{\rho}$. In such case, $\eta_{\Delta'}(\mathfrak h)=(1,0)$. And $\eta_{\Delta'}(\mathfrak r(\Delta, \mathfrak h))=\eta_{\Delta'}(\left\{ [3,8]_{\rho}, [4,5]_{\rho} \right\})=(1,1)$. 
\end{itemize}

\end{example}

A consequence of Proposition \ref{prop dagger eta condition} is the following:

\begin{corollary} \label{cor transitivity on dagger property}
Suppose the triple $(\Delta, \Delta', \mathfrak h)$ satisfies the non-overlapping property. Write $\Delta'=[a',b']_{\rho}$. Then, for any segment $\widetilde{\Delta}$ linked to $\Delta$ and of the form $[\widetilde{a},b']_{\rho}$ for $\widetilde{a}\geq a'$, the triple $(\Delta, \widetilde{\Delta}, \mathfrak h)$ also satisfies the non-overlapping property. 
\end{corollary}

\end{document}